\documentclass[11pt]{amsart}
\usepackage[utf8]{inputenc}
\usepackage{physics}
\usepackage{graphicx}
\usepackage{amsmath}
\usepackage{amssymb}
\usepackage{setspace}
\usepackage{textcomp}
\usepackage{extarrows}
\usepackage{ifthen}
\usepackage{pgfplots}
\usepackage{amsfonts}
\usepackage{amscd}
\definecolor{MFCB}{cmyk}{0,0.06,0.20,0.6}
\usepackage{mathtools}
\usepackage{amsrefs}
\usepackage{tikz}
\usepackage{pgfplots}
\usetikzlibrary{calc,math,backgrounds}
\usepackage{ifthen}
\usepackage{pgfplots}
\definecolor{MFCB}{cmyk}{0,0.06,0.20,0.6}

\usepackage{amsthm}
\newtheorem{thm}{Theorem}

\newtheorem{lemma}{Lemma}
\newtheorem{proposition}{Proposition}

\theoremstyle{definition}
\newtheorem{definition}{Definition}
\theoremstyle{remark}

\newcommand{\R}{\mathbb{R}}

\newcommand{\N}{\mathbb{N}}
\newcommand{\Z}{\mathbb{Z}}

\renewcommand{\norm}[1]{\left|#1 \right|}
\DeclareMathOperator{\supp}{Supp}
\DeclareMathOperator{\Div}{div}
\DeclareMathOperator{\DIV}{Div}
\DeclareMathOperator{\dist}{dist}

\DeclareMathOperator{\FSym}{FSym}
\DeclareMathOperator{\FAlt}{FAlt}
\pgfplotsset{compat=1.18}
\begin{document}

\title{On linear divergence in finitely generated groups}
\author[Letizia Issini]{Letizia Issini
}
\thanks{\textsc{Section de Mathématiques, Université de Genève, Switzerland \\ Email:} \texttt{letizia.issini@unige.ch}}

\begin{abstract}
In this paper, we show that wreath products of groups have linear divergence, and we generalise the argument to permutational wreath products. We also prove that Houghton groups $\mathcal{H}_m$ with $m\geq 2$ and Baumslag-Solitar groups have linear divergence. We explain how to generalise the argument for wreath products so that it holds for halo products of groups whose halo is large-scale commutative. Finally, we show that wreath products of graphs and Diestel-Leader graphs have linear divergence. The argument for Diestel-Leader graphs is further generalised to horocyclic products of proper, geodesically complete, Busemann $\delta$-hyperbolic spaces that are uniformly not a quasi-line.
\end{abstract}
\maketitle
\section{Introduction}The divergence of a geodesic metric space is a function that measures how difficult it is to connect two points, avoiding a certain ball around a third point. First mentioned by Gromov in \cite{Gr3} as the notion of divergence between two geodesics, it was formally defined as a quasi-isometry invariant of a geodesic metric space by Gersten \cite{Ger94b}.
It is at least linear, and by definition it is infinite for metric spaces that have more than one end. It is easy to check that direct products of non-compact metric spaces have linear divergence, and that it is exponential in one-ended hyperbolic metric spaces.

Divergence has been extensively studied in finitely generated groups. Certain classes of groups including lattices in semi-simple Lie groups of $\mathbb{Q}$-rank $1$ and $\mathbb{R}$-rank $\geq 2$, uniform lattices in higher rank semi-simple Lie groups \cite{DMS10}, and Thompson groups $F$, $T$ and $V$ \cite{GolanSapir}, have been shown to have linear divergence. 

The first example of a group with quadratic divergence is given in \cite{Ger94b} and it is a CAT(0)-group. 
Further examples of quadratic divergence include the mapping class group (\cite{Behrstock}, see also \cite{duchin2010divergence}).
Examples of groups with divergence polynomial of any degree are known, but also of divergence $n^{\alpha}$ and $log(n)n^{k}$ for $\alpha$ in a dense subset of $[2,\infty)$ and integers $k\geq 2$ (see \cite{brady2021divergence}). In \cite{OOS05}, elementary amenable groups with superlinear divergence, and examples of groups with superlinear arbitrarily close to linear divergence are given. Divergence is known to be at most exponential for finitely presented groups \cite{sisto2012metric}. 

It is interesting to understand which structural properties of a group are reflected in its divergence function. An important result of Dru\c{t}u, Mozes and Sapir is that linear divergence is equivalent to not having cut-points in the asymptotic cones \cite{DMS10}, and this can be used to show that non-virtually cyclic groups that satisfy a law have linear divergence (see \cite{DMS10} and \cite{DS05}). 

In this paper, our aim is to enlarge the class of groups known to have linear divergence. Our main result is that a wreath product $H\wr F$ of finitely generated groups has linear divergence (see Section \ref{section:wreath products}, Theorem \ref{proposition:wreath}). We extend the proof further to permutational wreath products $H\wr_X F$, where the action of $F$ on the set $X$ is quasi-transitive and either $H$ or $X$ is infinite (in Section \ref{section:permutational wreath product}, Theorem \ref{proposition:pwp}).
A similar argument is used to show that Houghton group $\mathcal{H}_2$, i.e $\FSym(\Z)\rtimes\Z$, has linear divergence. Like a wreath product over an infinite group, this is another example of semi-direct product with non-finitely generated normal subgroup, and the similarity of the proof lies in the use of commutation relations of elements that are ``far away", according to the lamplighter and lampshuffler interpretation respectively. The proof for $\mathcal{H}_2$ can be generalised to all lampshuffler groups $\FSym(H)\rtimes H$, but also Houghton groups $\mathcal{H}_m$ with $m\geq 2$ (see Section \ref{section:houghton groups}, Theorem \ref{houghton}). Lamplighters and lampshufflers are particular instances of halo products, as defined by Genevois and Tessera \cite{gen-tes}. With the assumption of being large-scale commutative, notion also defined in \cite{gen-tes}, we can show that halo products have linear divergence (Section \ref{section:halo products}, Theorem \ref{thm:halo products}). Large-scale commutativity is a necessary assumption: the examples constructed in \cite{OOS05} of elementary amenable lacunary hyperbolic groups are halo products. Next we show that Baumslag-Solitar groups have linear divergence (Section \ref{section:baumslag-solitar groups}, Proposition \ref{proposition:bsgroups}).
Note also that having linear divergence implies that the Floyd boundary consists of a point for any Floyd function (as follows from the definition of Floyd boundary, see \cite{floyd}). As a consequence, all these groups have Floyd boundary consisting of one point.
Moreover, having linear divergence, these groups have no cut-points in their asymptotic cones by \cite{DMS10}.

Going back to the more general context of divergence of geodesic metric spaces, we study Diestel-Leader graphs $DL(p,q)$, for $p,q\geq 2$, which are interesting examples of vertex-transitive graphs \cite{DL}. 
The wreath product $\Z_p\wr \Z$ has the Diestel-Leader graph $DL(p,p)$ as a Cayley graph (see \cite{woess}).
Therefore we can say that the divergence of Diestel-Leader graphs $DL(p,p)$ is linear. Diestel-Leader graphs with $p\neq q$ were the first example of connected locally finite vertex-transitive graphs which are not quasi-isometric to any Cayley graph \cite{eskin2006quasiisometries}.
We show that $DL(p,q)$ with $p\neq q$ also have linear divergence (Section \ref{section:Diestel-Leader graphs}, Theorem \ref{proposition:DL-graphs}). 
Both Diestel-Leader graphs and Baumslag-Solitar groups $BS(1,p)$ are instances of horocyclic products studied by Ferragut in \cite{ferragut-thesis}. We extend our results by proving that horocyclic products of proper, geodesically complete, Busemann $\delta$-hyperbolic spaces that are uniformly not a quasi-line have linear divergence (Section \ref{section:horocyclic products}, Theorem \ref{thm:horocyclic}).

Wreath products of graphs were defined by Erschler as a generalisation of wreath products of groups \cite{Erschler}. We show that the argument for linearity of the divergence function of wreath products of groups can be adapted to wreath products of graphs $B\wr A$, where the graphs $A$ and $B$ are vertex-transitive and locally finite. This can be generalised to wreath products $B\wr_I A$ with respect to a family of subsets of $A$ under some assumptions (Section \ref{section:wreath products of graphs}).

Let us summarise the results in a statement.
\vspace{2mm}
\def\theorem{\par\noindent{\bf Theorem\ } \ignorespaces}
\def\endtheorem{}

\begin{theorem}
    The following groups have linear divergence:
    \begin{itemize}
        \item wreath products $H\wr F$, where $H$ and $F$ are non-trivial finitely generated groups (Theorem \ref{proposition:wreath});
        \item in fact halo products whose halo is large-scale commutative (Theorem \ref{thm:halo products});

        \item permutational wreath products $H\wr_X F$ where $H$ and $F$ are non-trivial finitely generated groups and the action of $F$ on $X$ has finitely many orbits (Theorem \ref{proposition:pwp});
        \item Houghton groups $\mathcal{H}_m$, $m\geq 2$ (Theorem \ref{houghton});
        \item Baumslag-Solitar groups (Proposition \ref{proposition:bsgroups}).
    \end{itemize}
    Moreover Diestel-Leader graphs, and more generally horocyclic products of proper, geodesically complete, Busemann $\delta$-hyperbolic spaces that are uniformly not a quasi-line, have linear divergence (Theorem \ref{proposition:DL-graphs} and Theorem \ref{thm:horocyclic}).
\end{theorem}
\subsection*{Acknowledgements}The author was funded by the Swiss Government Excellence Scholarship No. 2022.0011 and acknowledges
partial support from Research Preparation Grant – Leading House Geneva (2022).

I am grateful to my supervisor Tatiana Nagnibeda for having suggested and discussed enthusiastically with me this topic of research. I also would like to thank Corentin Bodart for sharing some thoughts on the matter, and an anonymous referee for suggesting possible generalisations. Finally, I would like to thank Tom Ferragut and Davide Perego for discussions about horocyclic products of hyperbolic spaces.

\section{Definition of divergence} 
We use the definition of divergence from \cite{BD14}. For a geodesic metric space $X$, given three points $a$, $b$, $c\in X$ and parameters $\delta\in (0,1)$ and $\gamma\geq~0$, we define the divergence $\Div_\gamma(a,b,c;\delta)$ to be the infimum of lengths of paths that connect $a$ to $b$ outside $B(c,\delta r -\gamma)$, the open ball around $c$ of radius $\delta r-\gamma$, if this exists. We define it to be infinite otherwise. Here $r=~\min \{\dist(c,a),\dist(c,b)\}$. (If the radius is non-positive, we define the ball to be empty.)
\begin{definition}
We define $
\DIV_\gamma(n,\delta)$ to be \begin{equation}\DIV_\gamma(n,\delta)=\label{defdiv}\sup_{\begin{subarray}{l} a,b,c\in X, \\ \dist(a,b)\leq n \end{subarray}}\Div_\gamma(a,b,c;\delta). \tag{$\star$} \end{equation}
\end{definition}
The following equivalence relation on functions $f:\R_{+}\to\R_{+}$ provides the right setting to work with divergence.
\begin{definition}\label{def:eqrel}Given two functions $f,g:\mathbb{R}_{+}\to\mathbb{R}_{+}$ and $A\geq 1$, we say that $f\preceq _{A}g$ if $f(x)\leq Ag(Ax+A)+Ax+A$ for all $x\in\mathbb{R}_{+}$. We define an equivalence relation on such functions as $f\asymp_{A}g$ if $f\preceq _{A}g$ and $g\preceq _{A}f$. We write $f\preceq g$ (respectively $f\asymp g$) when there exists $A\geq 1$ such that $f\preceq_{A}g$ (respectively $f\asymp _{A} g$).
\end{definition}
We notice that if a space has more than one end, then its divergence is infinite. Therefore, in what follows, we restrict to one-ended spaces.

A metric space is \textbf{proper} if all its closed balls are compact and \textbf{periodic} if it is geodesic and there exists a ball whose orbit under $Isom(X)$ covers~$X$.

\begin{proposition}[Druţu, Mozes, Sapir, \cite{DMS10}]\label{welldef}
If a geodesic metric space $X$ is one-ended, proper, periodic, and there exist $\kappa\geq 0$ and $L\geq 1$ such that every point is at distance at most $\kappa$ from a bi-infinite $L$-bi-Lipschitz path, then $\DIV_{\gamma}(n,\delta)$ is independent of $\gamma$ and $\delta$ up to $\asymp$ for any $\delta\leq \frac{1}{1+L^{2}}$ and $\gamma\geq 4\kappa$. 

Moreover, under these conditions on the spaces, $\DIV_{\gamma}(n,\delta)$ is invariant under quasi-isometry up to $\asymp$. 
\end{proposition}
Let now $G$ be a finitely generated group. A choice of finite generating set $S$ makes $G$ into a geodesic metric space which is proper and periodic. We will assume throughout the paper that it is one-ended. The conditions of Proposition \ref{welldef} are satisfied with $L=1$ and $\kappa=\frac{1}{2}$. Therefore divergence is a quasi-isometry invariant. In particular it does not depend on the choice of $S$ and we can talk about the divergence of finitely generated groups.

If the space is a connected, one-ended, locally finite and vertex-transitive graph, then the conditions of Proposition \ref{welldef} are satisfied with $L=1$ and $\kappa=\frac{1}{2}$, and we can also fix the third point $c$ in the definition. As a result, $\DIV_{\gamma}(n,\delta)$ simplifies as $\sup\limits_{\begin{subarray}{l} a,b\in X, \\ \dist(a,b)\leq n \end{subarray}}\Div_\gamma(a,b,x_0;\delta)$, where $x_0\in X$ is a fixed vertex of the graph.

We can make a further simplification in formula $(\star)$ and take the supremum over points $a$ and $b$ such that $\norm{a}=\norm{b}$, where $\norm{x}=\dist(x,x_0)$ for any point $x$ in the graph. Indeed, if $\dist(a,b)\leq n$, and $\norm{b}\geq \norm{a}$, let $b'$ be at distance $\norm{a}$ from $x_0$ on a geodesic segment that connects $x_0$ to $b$. Then a shortest path from $a$ to $b$ outside $B(x_0,\delta r-\gamma)$ has the same length up to $\asymp$ as a shortest path from $a$ to $b'$ outside $B(x_0,\delta r-\gamma)$. Since $\dist(a,b)\leq n$, we have $\norm{b}\leq r+n$ and $\dist(b',b)=\norm{b}-r\leq n$; hence, $\dist (a,b')\leq 2n$ . Therefore, if we consider divergence up to $\asymp$, we can restrict to $\norm{a}=\norm{b}=r$, and get to the definition 
\begin{equation}\label{defsimpl1}
\DIV_\gamma(n,\delta)=
\sup_{\begin{subarray}{l} a,b\in X, \\ \dist(a,b)\leq n, \\ |a|=|b| \end{subarray}}
\Div_\gamma(a,b,x_0;\delta). 
\tag{$\star\star$} 
\end{equation}

Moreover, we can further modify the definition of divergence to 
\begin{equation}\label{defsimpl}
\DIV'_\gamma(n,\delta)=
\sup_{\begin{subarray}{l} a,b\in X, \\ |a|=|b|=n \end{subarray}}
\Div_\gamma(a,b,x_0;\delta). 
\tag{$\star\star\star$} 
\end{equation}
Let us show that $\DIV_\gamma(n,\delta)$ and $\DIV'_\gamma(n,\delta)$ are equivalent.
We can assume that $r\leq \frac{3n}{2}$ in $\DIV_\gamma(n,\delta)$ for $\Div_{\gamma}(a,b,x_0;\delta)$ not to be trivially linear by triangle inequality. But then $\Div_\gamma(a,b,x_0;\delta)$ where $\dist(a,b)\leq n$ and $\norm{a}=\norm{b}=\omega n$ with $\omega\leq \frac{3}{2}$ is considered in 
$\DIV'_\gamma(\omega n,\delta)$, and so $\DIV_\gamma(n,\delta)\preceq \DIV'_\gamma(n,\delta)$. Similarly, $\Div_\gamma(a,b,x_0;\delta)$ with $|a|=|b|=n$ is considered in $\DIV_\gamma(2 n,\delta)$, and so $\DIV'_\gamma(n,\delta)\preceq \DIV_\gamma(n,\delta)$.

From now on, we will use (\ref{defsimpl}) as definition for $\DIV_\gamma(n,\delta)$.
In the case of groups, the group identity $1$ will serve as the fixed point $x_0$. 

In Sections \ref{section:houghton groups}, \ref{section:halo products} and \ref{section:baumslag-solitar groups} we will make use of the following lemma.
\begin{lemma}\label{Lemma:rays}Let $X$ be a geodesic metric space and $\mathfrak{p}$ a bi-infinite geodesic through a point $x$ such that $x\notin {B}(y,\epsilon)$ for some $y\in X$ and $\epsilon >0$. Then not both the positive and negative rays of $\mathfrak{p}$ intersect $B(y,\frac{\epsilon}{2})$.
\end{lemma}
\begin{proof}
    Suppose for a contradiction that there are $a,b$ in different rays of $\mathfrak{p}$ such that $a$ and $b$ are in $B(y,\frac{\epsilon}{2})$. So $\dist(a,b)<\epsilon$ by triangle inequality, which implies $\dist(a,x)<\frac{\epsilon}{2}$ or $\dist(b,x)<\frac{\epsilon}{2}$, being $\mathfrak{p}$ a geodesic. Suppose without loss of generality that $\dist(a,x)<\frac{\epsilon}{2}$. Consider the geodesic triangle $[a,x,y]$. By triangle inequality, we have $\dist(x,y)\leq\dist(x,a)+\dist(a,y)< \epsilon$; but $x\notin {B}(y,\epsilon)$, yielding a contradiction.
\end{proof}
\section{Divergence function of wreath products}\label{section:wreath products}
In this section, we study the divergence function of wreath products of groups $H\wr F$, showing that is it linear. Indeed, this was known, for example, for the lamplighter group $\Z_{2}\wr \Z$, since it is solvable and all finitely generated non-virtually cyclic solvable groups have linear divergence (\cite{DMS10} and \cite{DS05}).

Given two groups $H$ and $F$, the \textbf{wreath product} $H\wr F$ is the semidirect product $\big{(}\bigoplus\limits_{F}H\big{)}\rtimes F$, where the action of $F$ on $\bigoplus\limits_{F}H=H^F$ is given by $$f\cdot k(x)=({}^{f}k)(x):=k(f^{-1}x)$$ and the product is $(k,f)(k',f')=(k {}^{f}k', ff')$, for all $x,f, f'\in F$ and $k,k'\in~\bigoplus\limits_{F}H$.
We assume that the groups $H$ and $F$ are non-trivial.
\begin{thm}\label{proposition:wreath}
If $H$ and $F$ finitely generated groups, then the wreath product $G=H\wr F$ has linear divergence. 
\end{thm}
Let $S$ and $T$ be finite symmetric generating sets for $H$ and $F$ respectively. Let $\delta_h^f:F\to H$ be the element of $H^F$ that takes the value $h$ at $f$ and $1_H$ elsewhere, for any $h\in H$ and $f\in F$, and let $\delta_h:=\delta_h^{1_F}$. Consider $\{(1,t)\mid ~t\in ~T\}\cup~\{(\delta_s,1_F)\mid s\in S\}$ as generating set for $G$, where $1:F\to H$ takes value $1_H$ everywhere in $F$. 

To simplify the notation, we write $f$ for $(1,f)$ and $h$ for $(\delta_h,1_F)$, for any $f\in F$ and $h\in H$ . Then the chosen generating set for $G$ is $T\cup S$. Define the metrics $\norm{.}_{H}$, $\norm{.}_{F}$ and $\norm{.}_{G}$ as $\norm{.}_{S}$, $\norm{.}_{T}$ and $\norm{.}_{T\cup S}$ respectively. We will sometimes call \textit{base group} the group $F$, \textit{lamp group} the group $H$ and \textit{lamps} the elements of $F$. We see elements of $H\wr F$ as moving a cursor through some lamps in $F$ and changing their \textit{brightness}, according to the elements of $H$. For example, the product $fh$ corresponds to $(1,f)(\delta_h, 1_F)=(\delta_h^f,f)$ and means reaching the element $f$ in the base group $F$ and then applying $h$ to that lamp.

An element $g$ of $G$ is a pair $(k,f)$, where $k:F\to H$ is a finitely supported function  and $f\in F$. Let $\{x_1,\dots, x_m\}\subseteq F$ be the support of $k$, and let $h_i$ be $k(x_i)\in H$. 
Then $k=\delta_{h_1}^{x_1}\dots \delta_{h_m}^{x_m}$ and our element $g$ can be written as $x_1 h_1 (x_1)^{-1}x_2 h_2 (x_2)^{-1}\dots x_m h_m (x_m)^{-1}
f$.
This corresponds to reaching the $F$-positions (\textit{lamps}) where we want to make some steps along $H$ (\textit{that we want to turn on}), making some $H$-moves  at those positions, and finally going to the desired final $F$-position. 
Any way of writing $g$ in terms of the generators $T\cup S$ consists of a walk in $F$ through the lamps that we wish to turn on, interrupted by $H$-steps at the lamps, and a final path in $F$ to reach the terminal lamp. Amongst such representations of $g$, there will be an optimal one (see \cite{Parry}), which we denote $e_1 h_1 e_2 h_2\dots e_m h_m e_{m+1}$, where $e_i$ is a path in $F$, $h_i\in H\setminus\{1\}$ for all $i$, and only $e_1$ and $e_{m+1}$ can possibly be the identity. By \textit{optimal} we mean that the walk in $F$ is of minimal length such that it passes through $x_1,\dots, x_m$ and ends at $f$. Therefore, $\norm{g}_G=\sum\limits_{i=1}^{m+1} \norm{e_i}_F +\sum\limits_{j=1}^{m}\norm{h_j}_H$. Also, $f=e_1 e_2\dots e_{m+1}$ and $k=\delta_{h_1}^{x_1}\dots \delta_{h_m}^{x_m}$, where $x_i=e_1 e_2\dots e_i$.

In the proof, we will use the following fact.
\begin{lemma}\label{lemma:lamps}If an element $g\in G$ has the lamp $x\in F$ on, then $\norm{g}_G\geq \norm{x}_F$.
\end{lemma}
This is because we need to reach the lamp in order to switch it on.
\begin{proof}[Proof of Theorem \ref{proposition:wreath}]
If $F$ is finite, for the wreath product $G$ to be infinite, $H$ has to be infinite. In that case, $G$ has a direct product of infinite groups as subgroup of finite index, so it has linear divergence. We assume now that $F$ is infinite.

We are going to show that any two elements $g$ and $g'$ at $G$-distance $n$ from the identity  
can be connected by a path of length linear in $n$ that avoids $B_G(1_G,\frac{n}{6})$. We may fix an element $g^{*}\in G$ with $\norm{g^{*}}_G=n$ and show that any $g\in G$ of length $n$ can be connected to $g^{*}$ by a path of length linear in $n$ outside $B_G(1_G,\frac{n}{6})$. Therefore, for any $g,g'\in G$ of length $n$, we can concatenate such paths and get a path from $g$ to $g'$ that avoids $B_G(1_G,\frac{n}{6})$ and whose length is linear in $n$.

Let $f^{*}$ be a fixed element in $F$ such that $\norm{f^{*}}_F=n-1$, and let $f_1,\dots, f_{n-1}$ be elements of $T$ such that $f^{*}=f_1 \dots f_{n-1}$. 
By our choice of generating set for $G$, it follows that $\norm{f^{*}}_G=n-1$. Let $h_0$ be a fixed generator of $H$ (so $\norm{h_0}_H=1$). Then $g^{*}:=f^{*}h_0$ is a fixed element of $G$ with $\norm{g^{*}}_G=n$. We will prove that any $g\in G$ of length $n$ can be connected linearly to $g^{*}$ outside $B_G(1_G,\frac{n}{6})$.

As above, we see $g=(k,f)$ as a walk on the base group $F$ where we alternate steps in $F$ with \textit{turning on lamps} in $H$ and we have a final walk in $F$ to reach the final position, write it as $g=e_1 h_1 e_2 h_2\dots e_m h_m e_{m+1}$, and assume that the walk in $F$ is optimal, so that $\norm{g}_G=\sum\limits_{i=1}^{m} \norm{e_i}_F +\sum\limits_{j=1}^{m+1} \norm{h_j}_H$.

We divide the proof of the theorem into several cases, depending on $g$.

\begin{enumerate}
\item[Case 1] If $f^{*}$ is one of the lamps that are already on in $g$, for metric reasons it is the only one and $g=f^{*}h$ for some generator $h$ of $H$. We simply change its brightness to $h_0$ if needed (see Figure \ref{Case1}). The path is $g=f^{*}h, f^{*}, f^{*}h_0$.
\item[Case 2]\label{case:easy} If the lamp $f^{*}$ is not on, and one of the lamps that $g$ has on (say $x_j=~e_1\dots e_j$) has $F$-metric at least $\frac{n}{6}$, we can move in $F$ without entering the ball $B_G(1_G,\frac{n}{6})$ by Lemma \ref{lemma:lamps}. We start by reaching $f^{*}$, then we apply $h_0$ (see Figure \ref{Case2}); next we turn the lamps of $g$ off.

More precisely, starting from $g$, we go back along $f$, which takes at most $n$ steps and brings the cursor to the identity. Then reach $f^{*}$ via $f_1,\dots, f_{n-1}$. Apply $h_0$, go back to the identity, then reach $f$, and apply $g^{-1}$. Finally multiply by $f^{*}$, to reach $g^{*}$. This takes at most $6n$ steps.  
The path is the following (starting from $g$ and given by subpaths in $F$ and $H$, which we interpret as being sequences of generators in $T$ and $S$ respectively):
\begin{flalign*}e_{m+1}^{-1},e_m^{-1},\dots, e_1^{-1},
f_1,\dots, f_{n-1}, h_0, f_{n-1}^{-1},\dots ,f_1^{-1},\\e_1, \dots , e_{m}, h_m^{-1}, e_m^{-1}, h_{m-1}^{-1},\dots ,h_1^{-1},  e_1^{-1}, f_1, \dots , f_{n-1}.
\end{flalign*}
The final element is $f^{*}h_0$, since its only lamp that is on is $f^{*}$, it has brightness $h_0$, and the final position of the cursor is $f^{*}$.
At any step of this path a lamp with $F$-length at least $\frac{n}{6}$ is on, so we are always outside $B_G(1_G, \frac{n}{6})$ by Lemma \ref{lemma:lamps}.

\item[Case 3]\label{case 3} If all the lamps that are on in $g$ have $F$-metric less that $\frac{n}{6}$, and $\norm{f}_F\geq\frac{n}{6}$, then apply $h_0$ to turn on the lamp $f$ (see Figure \ref{Case3}). If $f=f^{*}$, then for metric reasons only another lamp can be on in $g$ and $m=1$. We proceed by turning it off and reach $f^{*}h_0$ via $e_{2}^{-1}, h_1^{-1}, e_2$. 
If $f\neq f^{*}$, we can reduce to Case 2. 
This will take at most $6n$ steps. At each step (except the initial one) we have $f$ or $f^{*}$ on, meaning that we are outside $B_G(1_G,\frac{n}{6})$ by Lemma \ref{lemma:lamps}, since $\norm{f}_F,\norm{f^{*}}_F\geq\frac{n}{6}$.

\item[Case 4] If all the lamps that are on have $F$-metric less that $\frac{n}{6}$, and $\norm{f}_F<\frac{n}{6}$, we have that $\frac{n}{6}> \norm{f}_F=\norm{e_1 e_2\dots e_{m+1}}_F\geq \norm{e_{m+1}}_F-\norm{e_1 e_2 \dots e_{m}}_F$, so $\norm{e_{m+1}}_F< \norm{e_1 e_2 \dots e_{m}}_F+\frac{n}{6}<\frac{n}{3}$, since $e_1 e_2 \dots e_{m}$ is a lamp. 
Moreover, from $$n=\norm{g}_G=\sum\limits_{i=1}^{m} \norm{e_i}_F +\sum\limits_{j=1}^{m} \norm{h_j}_H+\norm{e_{m+1}}_F,$$ 
we have that $\sum\limits_{i=1}^{m} \norm{e_i}_F +\sum\limits_{j=1}^{m} \norm{h_j}_H=n-\norm{e_{m+1}}_F>\frac{2n}{3}$. 

This enables us to move in $F$ without getting too close to the identity in the $G$-metric. If we make a few steps in the base group $F$, say $w_1,\dots ,w_k$, where $w_i\in T$, we have that for each $1\leq i\leq k$, 
\begin{equation}\label{eq:ineq}
\norm{gw_1\dots w_i}_{G}\geq \sum\limits_{i=1}^{m} \norm{e_i}_F +\sum\limits_{j=1}^{m} \norm{h_j}_H-2\frac{n}{6}> \frac{n}{6},
\end{equation}
since the lamps of $g$ remain on, and $e_1,e_2,\dots ,e_{m+1}$ is a shortest walk that turns them on and reaches $f$. After moving in $F$, we are no longer reaching $f$, so our walk through the lamps might be shorter, but it cannot be shorter than $\sum\limits_{i=1}^{m} \norm{e_i}_F +\sum\limits_{j=1}^{m} \norm{h_j}_H$ minus the largest pairwise distance between the lamps.

We therefore proceed by reaching $1_F$ in $F$ and then $f^{*}$, and applying $h_0$ (see Figure \ref{Case4}).
The first part of the path is given by $$e_{m+1}^{-1},  e_{m}^{-1},\dots, e_{1}^{-1}.$$  
Then, we apply $f_{1},\dots, f_{n-1}, h_0$. 
Next, the path $f_{n-1}^{-1},\dots ,f_1^{-1},$\\$ e_1, h_1^{-1}, \dots, e_m, h_m^{-1}$ turns off all the other lamps. 
Finally, we go back to $f^{*}$ via $e_m^{-1}, \dots, e_1^{-1},f_1,\dots, f_{n-1}$ so that we reach the element $f^{*}h_0=g^{*}$.

The entire path has length at most $6n$, and is entirely outside  $B_G(1_G,\frac{n}{6})$ because of the inequality (\ref{eq:ineq}) first, and of $f^{*}$ being on later.
\end{enumerate}

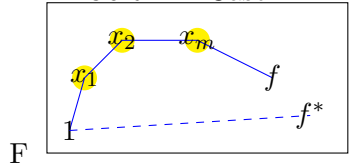
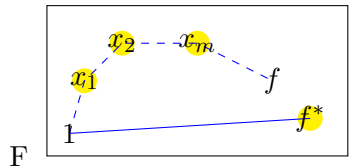
\begin{figure}\caption{Case 1: The path corresponds to changing brightness to the lamp $f^{*}$.}\label{Case1}
\begin{tikzpicture}
 \draw [draw=black] (0,2.5) rectangle (4,4.5);
 \node[text width=1cm] at (0,2.5)  {F};
\fill[color=orange!] (3.5,3) circle (1ex);
\draw [draw=blue] (0.3,2.8) node {$1$}-- (3.5,3) node {$f^{*}$};
 
\draw [draw=black] (0,0) rectangle (4,2);
\node[text width=1cm] at (0,0)  {F};
 
\fill[color=yellow!] (3.5,0.5) circle (1ex);

\draw [draw=blue] (0.3,0.3) node {$1$} -- (3.5,0.5) node {$f^{*}$};

\end{tikzpicture}
\end{figure}
\begin{figure}\caption{Case 2: The path consists in reaching lamp $f^{*}$ and turning it on. This path avoids $B_G(1_G,\frac{n}{6})$ as one of the lamps that $g$ has on has distance at least $\frac{n}{6}$ from $1_F$.}\label{Case2}
\begin{tikzpicture}
 \draw [draw=black] (0,2.5) rectangle (4,4.5) ;
 \node[text width=1cm] at (0,2.5) 
    {F};
\fill[color=yellow!] (0.5,3.5) circle (1ex);
\fill[color=yellow!] (1,4) circle (1ex);
\fill[color=yellow!] (2,4) circle (1ex);
\draw [draw=blue] (0.3,2.8) node {$1$}-- (0.5,3.5) node {$x_1$}
            -- (1,4) node {$x_2$} -- (2,4) node {$x_m$}-- (3,3.5) node {$f$};
\draw [dashed][draw=blue] (0.3,2.8) -- (3.5,3) node {$f^{*}$};
\node[text width=1cm] at (0,0)  {F};

\draw [draw=black] (0,0) rectangle (4,2);
 \fill[color=yellow!] (2,1.5) circle (1ex);
\fill[color=yellow!] (3.5,0.5) circle (1ex);
\fill[color=yellow!] (1,1.5) circle (1ex);
\fill[color=yellow!] (0.5,1) circle (1ex);
\draw [dashed][draw=blue] (0.3,0.3) node {$1$} -- (0.5,1) node {$x_1$}
            -- (1,1.5) node {$x_2$} -- (2,1.5) node {$x_m$}-- (3,1) node {$f$};
\draw [draw=blue] (0.3,0.3) -- (3.5,0.5) node {$f^{*}$};

\end{tikzpicture}
\end{figure}

\begin{figure}\caption{Case 3: The path consists in turning on lamp $f$, and then reaching $f^{*}$ and turning it on. This path avoids $B_G(1_G,\frac{n}{6})$ since $\norm{f}_F\geq \frac{n}{6}$.}\label{Case3}
\begin{tikzpicture}
 \draw [draw=black] (0,2.5) rectangle (4,4.5);
 \node[text width=1cm] at (0,2.5)  {F};
\fill[color=yellow!] (0.5,3.5) circle (1ex);
\fill[color=yellow!] (1,4) circle (1ex);
\fill[color=yellow!] (2,4) circle (1ex);
\draw [draw=blue] (0.3,2.8) node {$1$} -- (0.5,3.5) node {$x_1$}
            -- (1,4) node {$x_2$} -- (2,4) node {$x_m$}-- (3,3.5) node {$f$};
\draw [dashed][draw=blue] (0.3,2.8) -- (3.5,3) node {$f^{*}$};
 
\draw [draw=black] (0,0) rectangle (4,2);
\node[text width=1cm] at (0,0)  {F};
 \fill[color=yellow!] (3,1) circle (1ex);
\fill[color=yellow!] (3.5,0.5) circle (1ex);
\fill[color=yellow!] (2,1.5) circle (1ex);
\fill[color=yellow!] (1,1.5) circle (1ex);
\fill[color=yellow!] (0.5,1) circle (1ex);

\draw [dashed][draw=blue] (0.3,0.3) node{$1$} -- (0.5,1) node {$x_1$}
            -- (1,1.5) node {$x_2$} -- (2,1.5) node {$x_m$}-- (3,1) node {$f$};
\draw [draw=blue] (0.3,0.3) -- (3.5,0.5) node {$f^{*}$};
\end{tikzpicture}
\end{figure}
\begin{figure}\caption{Case 4: The path consists in reaching lamp $f^{*}$ and turning it on. We keep all the other lamps on in order to avoid $B_G(1_G,\frac{n}{6})$.}\label{Case4}
\begin{tikzpicture}
 \draw [draw=black] (0,2.5) rectangle (4,4.5);
 \node[text width=1cm] at (0,2.5)  {F};
\fill[color=yellow!] (0.5,3.5) circle (1ex);
\fill[color=yellow!] (1,4) circle (1ex);
\fill[color=yellow!] (2,4) circle (1ex);
\draw [draw=blue] (0.3,2.8) node {$1$} -- (0.5,3.5) node {$x_1$}
            -- (1,4) node {$x_2$} -- (2,4) node {$x_m$}-- (3,3.5) node {$f$};
\draw [dashed][draw=blue] (0.3,2.8) -- (3.5,3) node {$f^{*}$};
 
\draw [draw=black] (0,0) rectangle (4,2);
\node[text width=1cm] at (0,0)  {F};
 \fill[color=yellow!] (2,1.5) circle (1ex);
\fill[color=yellow!] (3.5,0.5) circle (1ex);
\fill[color=yellow!] (1,1.5) circle (1ex);
\fill[color=yellow!] (0.5,1) circle (1ex);
\draw [dashed][draw=blue] (0.3,0.3) node{$1$}-- (0.5,1) node {$x_1$}
            -- (1,1.5) node {$x_2$} -- (2,1.5) node {$x_m$}-- (3,1) node {$f$};
\draw [draw=blue] (0.3,0.3) -- (3.5,0.5) node {$f^{*}$};
\end{tikzpicture}
\end{figure}
\end{proof}

\section{Divergence function of permutational wreath products}\label{section:permutational wreath product}
We take now two groups, $H$ and $F$, and we let $F$ act on a set $X$. 
The \textbf{permutational wreath product} $H\wr_X F$ is the semidirect product $\big{(}\bigoplus\limits_{X}H\big{)}\rtimes F$ with elements $(k,f)$, where $k:X\to H$ is a finitely supported function and $f\in F$. The action of $F$ on $\big{(}\bigoplus\limits_{X}H\big{)}=H^X$ is given by $f\cdot k(x)=({}^{f}k)(x)=k(f^{-1}\cdot x)$ for all $x\in X$, $f\in F$ and $k\in H^F$; the product is $(k,f)(k',f')=(k {}^{f}k', ff')$. This construction generalises the usual wreath product in the case when $F$ acts on a set other than itself. 

If $F$ and $H$ are finitely generated, and the $F$-action on $X$ is quasi-transitive (i.e. it has a finite number of orbits), then $G=H\wr_X F$ is finitely generated. This is because, if $T$ and $S$ are finite generating sets for $F$ and $H$ respectively, and $\{y_0,\dots, y_q\}\subseteq~X$ is a set of base points for the orbits, then $\{(\delta_{s}^{y_i},1_F)\mid~s\in~S, 0\leq i\leq~q\}\cup\{(1,t)\mid~t\in T\}$ is a generating set for $G$. Here $\delta_{h}^{x}:X\to H$ takes the value $h$ at $x$ and $1_H$ elsewhere on $X$, and $1:X\to H$ takes the value $1_H$ everywhere on $X$. We assume that $H$ is non-trivial and $|X|\geq 2$ (if $|X|=1$, we get $H\times F$). We have the following result.
\begin{thm}\label{proposition:pwp} If $H$ and $F$ are finitely generated groups, $F$ acts on a set $X$ with finitely many orbits, and $H$ or $X$ is infinite, then $H\wr_X F$ has linear divergence.
\end{thm}
We study first the case in which the action is transitive, and then generalise to quasi-transitive actions. Suppose that $F$ is acting transitively on $X$ and fix a base point $y_0$ of $X$.  
For the wreath product $G$ to be an infinite group, we need at least one of $F$, $H$ or $X$ to be infinite (under the transitivity assumption, if $X$ is infinite then $F$ is also infinite). In the case in which $X$ is finite, $H$ is finite and $F$ is infinite, $G$ has a finite index subgroup isomorphic to $F$, so it has the same divergence as $F$.

Just like in the previous section, we write $f$ for $(1,f)$ and $h$ for $(\delta_h^{y_0},1_F)$. Then the product $fh$ corresponds to $(\delta_{h}^{f\cdot y_0},f)$, that is we act on $y_0$ by $f$ and then apply $h$ to $f\cdot y_0$, while in the second coordinate we have $f$. 
Let $g$ be the element $(k,f)$ where $k:X\to H$ has support $\{x_1,\dots,x_m\}\subseteq X$ and $k(x_i)=h_i$. If $q_i\in F$ is such that $q_i\cdot y_0=x_i$, then we can write $g$ as $q_1 h_1 q_1^{-1}\dots q_m h_m q_m ^{-1} f$. We think of $k$ as a path in $X$ that starts at $y_0$, passes through each point in $\{x_1,\dots,x_m\}$ in some order and at each of them applies the relative $h_i$, and then goes back to $y_0$. Any element of $G$ can be represented as a walk in $X$ through the lamps that we want to turn on, with some $H$-steps at those lamps, and an element $f\in F$.

We fix $S$ and $T$, finite symmetric generating sets for $H$ and $F$ respectively. Define the metrics $\norm{.}_{H}$, $\norm{.}_{F}$ and $\norm{.}_{G}$ as $\norm{.}_{S}$, $\norm{.}_{T}$ and $\norm{.}_{T\cup S}$ respectively. We can see $X$ as a metric space with the metric $\norm{.}_X$ defined by $d_X(.,y_0)$, where $d_X(x,y)=\min\{\norm{f}_F\mid f\cdot x=y, f\in F\}$. This metric space is isometric to the Schreier graph of the $F$-action on $X$ with respect to the generating set $T$, $Sch(F,X,T)$, equipped with the graph metric. 

The argument is analogous to that of the previous section, except Case 4, since here it is possible that $\norm{f}_F$ is large but $\norm{f\cdot y_0}_X$ is small (for instance $f$ could be in the stabilizer of $y_0$). Like in the other case, we have the following fact.
\begin{lemma}\label{lemma:lamps2}If an element $g\in G$ has the lamp $x\in X$ on, then $\norm{g}_G\geq \norm{x}_X$.
\end{lemma}
\begin{proof}[Proof of Theorem \ref{proposition:pwp} (when the action is transitive)]
We assume first that $X$ is infinite. We use the same strategy as in the previous case and connect each element $g\in G$ with $\norm{g}_G=~n$ to a fixed element $g^{*}$ with $\norm{g^{*}}_G=~n$ via a path that avoids $B_G(1_G,\frac{n}{6})$ whose length is linear in $n$. Then, by concatenation, we will be able to connect any two elements of $G$ with $G$-length $n$ linearly outside $B_G(1_G,\frac{n}{6})$.

Let $x^{*}\in X$ be a point with $\norm{x^{*}}_X=n-1$, let $f_1,\dots ,f_{n-1}$ be generators of $F$ such that $f_1\dots f_{n-1}\cdot y_0=x^{*}$, and define $f^{*}=f_1\dots f_{n-1}$. Fix a generator $h_0$ of $H$. Define $g^{*}$ to be $f^{*}h_0=(\delta_{h_0}^{x^{*}},f^{*})$. Then $\norm{g^{*}}_G=n$ because the lamp at $x^{*}$ is on, and the smallest number of generators of $F$ we need to use in order to get to $x^{*}$ from $y_0$ is $n-1$. Fix also a point $x^{**}\in X$ with $\norm{x^{**}}_X=\lceil \frac{n}{6}\rceil$. 

Let $g=(k,f)\in G$ have $\norm{g}_G=n$. We write $g$ as a sequence of minimal length of generators as $e_1 h_1 \dots e_m h_m e_{m+1}$, where $e_i\in F$ and $h_i\in~H$ are sequences of generators in $T$ and $S$ respectively, and only $e_1$ and $e_{m+1}$ can possibly be $1$. We have that $g=(\delta_{h_1}^{e_1\cdot y_0} \delta_{h_2}^{e_1 e_2\cdot y_0} \dots \delta_{h_m}^{e_1 e_2 \dots e_m \cdot y_0} ,e_1 e_2 \dots e_{m+1})$, $f=~e_1 e_2 \dots e_{m+1}$, and $n=~\norm{g}_G=\sum\limits_{i=1}^{m+1} \norm{e_i}_F+\sum\limits_{i=1}^{m}\norm{h_i}_H$. Let $x_j=e_1 \dots e_j \cdot y_0$ for $1\leq j\leq m$ be the lamps that are on.
Let $t_0, \dots, t_{r-1}$ be generators of $F$ with $r$ minimal such that $t_0\dots t_{r-1}\cdot y_0=f\cdot y_0$. Let $f=s_1\dots s_l$ in terms of generators that realise its $F$-length, so $r\leq l$.

\begin{enumerate}
\item[Case 1]\label{trivial} If $x^{*}$ is one of the lamps that are already on in $g^{*}$, for metric reasons it is the only one and $g=f^{*}h$ for some generator $h$ of $H$. We simply change its brightness to $h_0$ if needed. The path is $g=f^{*}h, f^{*}, f^{*}h_0$.
\item[Case 2]\label{far} If one of the lamps that are on in $g$, say $x_j=e_1\dots e_j\cdot y_0$, is such that $\norm{x_j}_X\geq \frac{n}{6}$, we can start by reaching $x^{*}$ and turning it on. That is, we apply $$e_{m+1}^{-1}, \dots,e_1^{-1}, f_1, \dots, f_{n-1}, h_0,$$ reaching the element $(\delta_{h_1}^{x_1}\dots \delta_{h_m}^{x_m} \delta_{h_0}^{x^{*}},f_1\dots f_{n-1})$.

Then we turn off the lamps of $g$ and go back to $x^{*}$ via $$f_{n-1}^{-1},\dots ,f_1^{-1},e_1, \dots ,e_m, h_m^{-1}, e_m^{-1},\dots ,h_1^{-1},e_1^{-1}, f_1,\dots , f_{n-1},$$ reaching $g^{*}$.

This takes at most $6n$ steps, and at any point there is a lamp on with $X$-length at least $\frac{n}{6}$, so that we avoid $B_G(1_G,\frac{n}{6})$ by Lemma \ref{lemma:lamps2}.  
\item[Case 3] If all the lamps that are on are at $X$-distance less than $\frac{n}{6}$ from $y_0$ and $\norm{f\cdot y_0}_X\geq\frac{n}{6}$, then start by applying $h_0$ and get $(k\delta^{f\cdot y_0}_{h_0},f)$. If $f\cdot y_0=x^{*}$, then there is at most another lamp on for metric reasons, and $m=1$. We proceed by turning it off and reach $f^{*}h_0$ via $e_{2}^{-1}, h_1^{-1}, e_2$. 
If $f\cdot y_0\neq x^{*}$, we reduce to Case 2.
\item[Case 4] If all the lamps that are on are at $X$-distance less than $\frac{n}{6}$ from $y_0$ and $\norm{f\cdot y_0}_X< \frac{n}{6}$, we need to make some $F$-moves first. Let $a_1\dots a_p$ be an element of $F$ written in terms of a minimal number of generators such that $a_1\dots a_p\cdot y_0=x^{**}$. So $p=\lceil \frac{n}{6}\rceil$. Multiply $g$ by $t_{r-1}^{-1},\dots, t_0^{-1},a_1,\dots, a_p, h_0$ in this order. Since $r<\frac{n}{6}$ and $p=\lceil \frac{n}{6}\rceil$, at each step, the $G$-metric of the element will be at least $\frac{n}{6}$ by triangle inequality. We get the element $(k \delta^{x^{**}}_{h_0}, e_1\dots e_{m+1} t_{r-1}^{-1}\dots t_{0}^{-1}a_1\dots a_p)$, which has $x^{**}$ on. By keeping it on, we can turn off the other lamps, turn on $x^{*}$ and reach $g^{*}$. To do this, multiply by $$a_{p}^{-1},\dots, a_{1}^{-1}, t_0,\dots, t_{r-1}, e_{m+1}^{-1},h_m^{-1},\dots ,e_{1}^{-1},$$ reaching the element $(\delta^{x^{**}}_{h_0},1_F)$. Then multiply by $f_1,\dots ,f_{n-1},h_0$, to turn on the lamp at $x^{*}$, reaching $(\delta^{x^{**}}_{h_0}\delta^{x^{*}}_{h_0},f_1 \dots f_{n-1})$. Finally, switch the lamp at $x^{**}$ off, via $$f_{n-1}^{-1},\dots, f_1^{-1}, a_1, \dots ,a_p, h_0^{-1},a_{p}^{-1},\dots, a_{1}^{-1},f_1,\dots ,f_{n-1}.$$ The element that we reach is $g^{*}$. 

The number of steps this takes is at most $6n$ and at each step we are outside $B_G(1_G,\frac{n}{6})$, by triangle inequality first and by having $x^{**}$ or $x^{*}$ on later.

\end{enumerate}
In any case, we can reach the element $g^{*}$ in at most $6n$ steps avoiding $B_G(1_G,\frac{n}{6})$.

If $X$ and $F$ are finite, and $H$ is infinite, then $G$ is virtually a direct product of infinite groups, so it has linear divergence.

In the case in which $X$ is finite, and $H$ and $F$ are infinite, we also have linear divergence. Let $X$ be $\{x_0,\dots, x_l\}$, with $y_0=x_0$, let $c_i$ be a shortest element in $F$ such that $c_i\cdot y_0=x_i$, and let $M=max \norm{c_i}_F$. Let $g=(k,f)\in G$ have $\norm{g}_G=n$ and define $h_i=k(x_i)$. Since $|X|\geq 2$, we have that $\norm{h_j}_H\leq \frac{n}{2}$ for some $j$. Let $h\in H$ be a fixed element of $H$ such that $\norm{h}_H=n$. We can connect $g$ in linear time to $(\delta_h^{y_0},1_F)$ avoiding $B_G(1_G,\frac{n}{2}-2M)$.

There is $f'\in F$ with $\norm{f'}_F\leq M$ such that $f'\cdot (f\cdot y_0)=y_0$. Multiply $g$ by $f' c_j$, and then by $h_j^{-1}h$ to reach $(k \delta_{h_j^{-1}}^{x_j}\delta_{h}^{x_j}, ff'c_j)$. By triangle inequality, at any step, the metric is at least $n-2M-\frac{n}{2}=\frac{n}{2}-2M$, since $\norm{f'}_F\leq M$, $\norm{c_j}_F\leq M$, $\norm{h_j}_H\leq \frac{n}{2}$ and the metric increases as we multiply by $h$.
Now, we turn off all the lamps except $x_j$ to reach $(\delta_h^{x_j},1_F)$. If $x_j=y_0$ we are done, otherwise we turn on $y_0$ by the amount $h$ while keeping $x_j$ on, and then we turn off $x_j$.
\end{proof}

We consider now permutational wreath products where the action of $F$ on $X$ has a finite number of orbits. If $F$ is finite, then $X$ is also finite since we have finitely many orbits. For $G$ to be infinite, $H$ has to be infinite. Hence $G$ has the same divergence as the direct sum of copies of $H$, which is linear.

If $F$ is infinite, and $X$ is finite, the divergence is the same as that of $F$ if $H$ is finite, and it is linear if $H$ is infinite, by the same argument as in the transitive case.

Now suppose that $F$ is infinite and $X$ is infinite. 
We choose a base point for each orbit and see each orbit as a metric space where the distance between two points $x$ and $y$ is $\min\{\norm{f}_F\mid f\cdot x=y, f\in F\}$, i.e. the length of a shortest path between $x$ and $y$ in the Schreier graph of the $F$-action on that orbit with respect to $T$. We can use the same arguments as before, applying them to one of the orbits, say $X_0$ with basepoint $y_0$. Fix elements $x^{*}$ and $x^{**}$ in $X_0$ of $X_0$-length $n-1$ and $\lceil \frac{n}{6}\rceil$ respectively, and $h_0$ a generator of $H$, and connect any $g$ of $G$-length $n$ to the element that has $h_0$ at $x^{*}$.
Once we have turned on lamp $x^{*}$, we turn off those of $g$ in order to get to an element that only has the lamp $x^{*}$ on.
\section{Divergence function of wreath products of graphs}\label{section:wreath products of graphs}
Given two graphs $A$ and $B$, we can construct the \textbf{wreath product} $\Gamma=~B\wr A$, as defined in \cite{Erschler} (note that what Erschler defines as $A\wr B$ is here $B\wr A$, for the definition to be consistent with the notation in the case of Cayley graphs). If $A$ and $B$ are Cayley graphs of groups, then $B\wr A$ is the Cayley graph of the wreath products of the groups with respect to the generating set used in Section \ref{section:wreath products}.

Let $b_0$ be a fixed vertex in $B$. We will consider functions $f:A\to B$, where $A$ and $B$ are just the vertex sets of $A$ and $B$, and the support of $f$ is $\supp(f)=\{a\in A\mid f(a)\neq b_0\}$. The vertices of the graph $B\wr A$ are $(f,a)$, where $a\in A$ and $f: A\to B$ has finite support. The edges can be of two types. Those of type (1) connect vertices $(f_1,a_1)$ and $(f_2,a_2)$ where $a_1=a_2$, $f_1(x)=f_2(x)$ for all $x\neq a_1$ and $f_1(a_1)$ is connected to $f_2(a_1)$ by an edge in $B$; those of type (2) connect vertices $(f_1,a_1)$ and $(f_2,a_2)$, where $f_1(x)=f_2(x)$ for all $x\in A$ and $a_1$ is connected to $a_2$ by an edge in $A$.

We assume that $A$ and $B$ are connected, and that $|A|\geq 2$ and $|B|\geq 2$. In the case where $A$ and $B$ are vertex-transitive, the wreath product $B\wr A$ is also vertex-transitive, and the conditions of Proposition \ref{welldef} are satisfied. 
\begin{thm} If $A$ and $B$ are connected, vertex-transitive, locally finite graphs, then their wreath product $B\wr A$ has linear divergence.
\end{thm}
\begin{proof}If $A$ is finite and $B$ is infinite, the wreath product is quasi-isometric to a direct product of infinite graphs, which can be easily checked to have linear divergence. We assume now that $A$ is infinite.
    
    Because $B\wr A$ is vertex-transitive, we may fix the point $x_0$ in the definition of divergence, and use Definition (\ref{defsimpl}). Let $a_0$ be a fixed point in $A$, and let $f_0:A\to B$ be the constant map $a\mapsto b_0$. Let $g_0:=(f_0,a_0)$ be our fixed point $x_0$. We define the metrics in the graphs $A$, $B$, and $\Gamma$ by $\norm{a}_A:=\dist_A(a,a_0)$,  $\norm{b}_B:=\dist_B(b,b_0)$, and $\norm{g}_\Gamma:=\dist_\Gamma(g,g_0)$ for $a\in A$, $b\in B$, and $g\in\Gamma$. If an element $g=(f,a)$ is such that $f(x)\neq b_0$ for some $x\in A$, then $\norm{g}_\Gamma\geq \norm{x}_A$.

    Using the same strategy as in the group case, we will connect every element of $\Gamma$ with metric $n$ to $g^{*}$, a fixed element of $\Gamma$ with metric $n$, using a path of length linear in $n$ and avoiding $B_\Gamma(g_0, \frac{n}{6})$. We fix vertices $b^{*}\in B$ and $a^{*}\in A$ with $\norm{b^{*}}_B=1$ and $\norm{a^{*}}_A=n-1$. Let $g^{*}=(f^{*},a^{*})$, where $f^{*}(a^{*})=b^{*}$ and $f^{*}(a)=b_0$ for all $a\neq a^{*}$. It is easy to see that $\norm{g^{*}}_\Gamma=n$.

    Now consider a vertex $g=(f,a)\in \Gamma$ with $\norm{g}_\Gamma=n$. If $f(a^{*})\neq b_0$, then we can connect $g$ to $g^{*}$ using at most two edges of type (1).

    If $f(a^{*})= b_0$ and there is $x\in A$ with $\norm{x}_A\geq\frac{n}{6}$ such that $f(x)\neq b_0$, we move to $(f,a^{*})$ using edges of type (2), then we use an edge of type (1) to get to $(f', a^{*})$, where $f'(a^{*})=b^{*}$ and $f'(a)=f(a)$ for all $a\neq a^{*}$. Then, using edges of type (2), we move in $A$ through the support of $f$, and at each vertex $y$ of the support we move along edges of type (1) from $f(y)$ to $b_0$. Once we are done, we use edges of type (2) to reach $a^{*}$. The element we are left with is $g^{*}$. For each vertex $(\phi,\alpha)$ we pass through,  $\phi(x)\neq b_0$, or $\phi(a^{*})\neq b_0$, so we are always outside $B_\Gamma(g_0,\frac{n}{6})$.

    If $\norm{x}_A<\frac{n}{6}$ for all $x\in \supp(f)$, and $\norm{a}_A\geq \frac{n}{6}$, we start by moving along an edge of type (1) and reach $(f', a)$, where $f'(a)=b^{*}$ and $f'(x)=f(x)$ for all $x\neq a$. Then we can reduce to the previous cases. 

    If $\norm{x}_A<\frac{n}{6}$ for all $x\in \supp(f)$, and $\norm{a}_A< \frac{n}{6}$, we have that, if we make moves of type (2), at each step the metric is at least $\frac{n}{6}$ for the same reason as in the group case. So we proceed by reaching $(f,a^{*})$ and then reducing to the previous case.

In each case, we were able to produce a path in $\Gamma$ from $g$ to $g^{*}$, with a number of edges that is linear in $n$, and that avoids the ball $B_\Gamma(g_0,\frac{n}{6})$.
\end{proof}
\subsection{Divergence function of wreath products of graphs with respect to a family of subsets}Wreath products of graphs with respect to a family of subsets were introduced in \cite{Erschler} and they are a way to interpolate between direct products and wreath products of two graphs, for which it is known that divergence is linear. Given graphs $A$ and $B$ and a family of subsets $\{A_i\mid i\in I\}$ of $A$, we define the wreath product $B\wr_I A$ with respect to this family as the graph whose vertices are pairs $(f,a)$, where $a\in A$ and $f:I\to B$ is a function with finite support.
The support of a function $f:I\to B$ is $\{i\in I\mid f(a)\neq b_0\}$, where $b_0$ is a fixed vertex in $B$. 
The edges can be of two types. Those of type (1) connect vertices $(f_1,a_1)$ and $(f_2,a_2)$ where $a_1=a_2=a$, there exists $i_a$ with $a_1\in A_{i_a}$ such that $f_1(i)=f_2(i)$ for all $i\neq i_a$, and $f_1(i_a)$ is connected to $f_2(i_a)$ by an edge in $B$; those of type (2) connect vertices $(f_1,a_1)$ and $(f_2,a_2)$, where $f_1(x)=f_2(x)$ for all $x\in I$ and $a_1$ is connected to $a_2$ by an edge in $A$.

Assume that $A$ and $B$ are connected, and that $|B|\geq 2$ and $|I|\geq 2$ (if $|I|$=1, then we get $B\times A$). Assume that the chosen family of subsets partitions $A$ and that the family is transitive, that is for all $a_1$, $a_2\in A$ there is an isometry $\phi_{a_1,a_2}$ that preserves the partition and sends $a_1$ to $a_2$ (this condition implies that $A$ is vertex-transitive). Assume moreover that $B$ is vertex-transitive. Then $\Gamma=B\wr_I A$ is vertex-transitive (see Lemma 2.3(iv) in \cite{Erschler}).
Therefore, the conditions of Proposition \ref{welldef} are satisfied and we can apply the same procedure as in previous cases. 

Let $a_0$ be a fixed vertex in $A$, and define $g_0=(f_0,a_0)$, where $f_0(i)=b_0$ for all $i\in I$. Define $\norm{.}_A$, $\norm{.}_B$ and $\norm{.}_\Gamma$ as above. If both $I$ and $B$ are finite, $B\wr_I A$ is quasi-isometric to $A$, so it has the same divergence as $A$. If $I$ and $A$ are finite, and $B$ is infinite, then $B\wr_I A$ is quasi-isometric to a direct product of $|I|$ copies of $B$, so it has linear divergence. If $I$ is finite, and $A$ and $B$ are infinite, we can mimic the corresponding argument in the group case, that is, taken $g\in \Gamma$, we change $f(j)$ such that $\norm{f(j)}_B\leq\frac{n}{2}$ to a fixed amount of $B$-length $n$ staying outside the ball $B(g_0,\frac{n}{2}-2M)$ for some constant $M$ in a linear number of steps.

The most involved case is the one with $I$ (and therefore $A$) infinite. 
We argue in a similar way as in the group case. We define a graph $\mathcal{I}$ as the quotient graph of $A$ by the equivalence relation given by $I$ (and we forget about the edges that connect two vertices in the same $A_i$). That is, the vertices of $\mathcal{I}$ are the elements of $I$, and two vertices $i$ and $j$, with $i\neq j$, are connected by an edge in $\mathcal{I}$ if and only if an element of $A_i$ is connected to an element of $A_j$ by an edge in $A$, and there are no edges from a vertex to itself. 

\begin{proposition}
    Let $A$, $B$ be connected, locally finite graphs, and $\{A_i\mid ~i\in ~I\}$ a transitive family of subsets that partitions $A$. If $B$ is vertex-transitive, $I$ is infinite, and $\mathcal{I}$ is locally finite, then $B\wr_I A$ has linear divergence.
\end{proposition}

\begin{proof}

We proceed by fixing $i_0$ such that $a_0\in A_{i_0}$. Using the graph metric in $\mathcal{I}$, define $\norm{i}_\mathcal{I}:=\dist_\mathcal{I}(i, i_0)$ for any $i\in I$. Suppose that $A_{i_1},\dots ,A_{i_k}$ are the subsets of the family adjacent to $A_{i_0}$, and let $x_j$ be an element of $A_{i_j}$ closest to $a_0$. Let $K=\max\limits_{1\leq j\leq k} \norm{x_j}_\mathcal{I}$. 
Then for any other vertex $x\in A$, there are $k$ sets of the family adjacent to $A_x$ and they have representatives at distance $\norm{x_1}_\mathcal{I},\dots,\norm{x_k}_\mathcal{I}$ from $x$. That is because $\phi_{a_0,x}(x_j)$ are all in different sets (also different from $A_x$), since $\phi_{a_0,x}$ preserves the partition, and they represent all the sets adjacent to $A_x$, and so there cannot be a vertex at a different distance from $x$ than $\norm{x_1}_\mathcal{I},\dots, \norm{x_k}_\mathcal{I}$ that is a closest vertex of a subset to $x$.

Therefore, given a path in $\mathcal{I}$ of length $l$ starting at $i_0$, say $i_0,i_1,\dots ,i_l$, we can produce a path in $A$ of length at least $l$ and at most $Kl$ starting at $a_0$ that contains points $a_1,\dots ,a_l$, where $a_j\in A_{i_{j}}$, in this order.

Fix $i^{*}, i^{**}\in I$ with $\norm{i^{*}}_\mathcal{I}=n-1$ and $\norm{i^{**}}_\mathcal{I}=\lceil \frac{n}{6K}\rceil$. Fix an element $a^{*}\in A$ with $n-1\leq \norm{a^{*}}_A\leq K(n-1)$ such that $A_{i_{a^{*}}}=A_{i^{*}}$. 
Fix also $b^{*}\in B$ with $\norm{ b^{*}}_B=1$. Define $g^{*}=(f^{*},a^{*})$, with $f^{*}(i^{*})=b^{*}$ and $f^{*}(i)=b_0$ for all $i\neq i^{*}$.

We take $g=(f,a)\in \Gamma$ with distance $n$ from $g_0$ and the goal is to produce a path of length linear in $n$ that connects it to $g^{*}$ outside $B_\Gamma(g_0,\frac{n}{6K})$. We use the fact that, if an element $g=(f,a)$ is such that $f(i)\neq b_0$ for some $i\in \mathcal{I}$, then $\norm{g}_\Gamma\geq \norm{i}_\mathcal{I}$.

If $f(i^{*})\neq b_0$, then we can use at most two edges of type (1) to make $f(i^{*})$ into $b^{*}$.

If there is $j\in \supp(f)$ with $\norm{j}_\mathcal{I}\geq \lceil \frac{n}{6K}\rceil$, then using moves of type (2) we move through $A$ and reach $(f,a^{*})$. We make a move of type (1) to change $f(i^{*})$ to $b^{*}$.
Then we move through the support of $f$ and, applying moves of type (1) to elements of the support of $f$, we make $f(y)$ into $b_0$ for all $y\in \supp(f)$. Finally, move back to $a^{*}$ using edges of type (2). We obtain $(f^{*},a^{*})$. 

If $\norm{i}_\mathcal{I}< \lceil \frac{n}{6K}\rceil$ for all $i\in \supp(f)$, but $\norm{i_a}_\mathcal{I}\geq \lceil \frac{n}{6K}\rceil$, we make a move of type (1) and change $f(i_a)$ from $b_0$ to $b^{*}$, and we reduce to the previous case.

Finally, if $\norm{i}_\mathcal{I}< \lceil \frac{n}{6K}\rceil$ for all $i\in \supp(f)$, and $\norm{i_a}_\mathcal{I}< \lceil \frac{n}{6K}\rceil$, we use moves of type (2) to reach $i_0$ and then $i^{**}$, that is we move along a path in $A$ from $a$ to some $a'\in A_0$ that is at distance at most $K\lceil \frac{n}{6K}\rceil$ from $a$, and then from $a'$ to some $a''\in A^{**}$ at distance at most  $K\lceil \frac{n}{6K}\rceil$ from $a'$.
Then using a move of type (1), we change $f(i^{**})$ from $b_0$ to $b^{*}$. Next, reach $a^{*}$ using moves of type (2), and use the same argument as in the group case.
\end{proof}
\section{Divergence function of Diestel-Leader graphs}\label{section:Diestel-Leader graphs}
The Diestel-Leader graph $DL(p,q)$, with $p, q\geq 2$, is the horocyclic product of homogeneous trees $\mathbb{T}_1$ and $\mathbb{T}_2$ of degrees $p+1$ and $q+1$. We follow \cite{woess} for their description.

Let $\mathbb{T}$ be a homogeneous tree of degree $p+1$, where $p\geq 2$. A geodesic ray $(x_n)$ is a one-sided infinite sequence of vertices of $\mathbb{T}$ such that $\dist(x_i,x_j)=~|j-~i|$ in the tree. Two rays $(x_n)$ and $(y_n)$ are equivalent if there are $N$ and $m$ such that $x_i=y_{i+m}$ for all $i\geq N$. An \textbf{end} is an equivalence class of geodesic rays, and $\partial\mathbb{T}$ is the space of ends. Let $o\in\mathbb{T}$ be a fixed root and $\omega$ a fixed end. Given vertices $a,b\in\mathbb{T}$, their confluent with respect to $\omega$ is $c=a\curlywedge b$ such that $\overline{c\omega}=\overline{a\omega}\cap \overline{b\omega}$, where, given $x\in\mathbb{T}$ and $\gamma\in\partial \mathbb{T}$, $\overline{x\gamma}$ is the ray that starts at $x$ and represents $\gamma$.

The \textbf{Busemann function} $\mathfrak{h}:\mathbb{T}\to\mathbb{Z}$ is $\mathfrak{h}(x)=\dist(x,x\curlywedge o)-\dist(o,x\curlywedge o)$. The \textbf{horocycles} are $H_{k}=\{x\in\mathbb{T}\mid \mathfrak{h}(x)=k\}$. For a vertex $x\in H_k$, its \textbf{parent} is the unique $y\in H_{k-1}$ that $x$ is connected to by an edge. Its \textbf{children} are the vertices in $H_{k+1}$ that are connected to $x$ by an edge; there are $p$ of them. The Busemann function is a \textit{height} function on the homogeneous tree, and the \textbf{horocycles} are levels of the tree with respect to such height.

Now, we take homogeneous trees $\mathbb{T}_1$ and $\mathbb{T}_2$ of degrees $p+1$ and $q+1$, we fix roots $o_1$ and $o_2$ respectively, and ends $\omega_1$ and $\omega_2$ respectively (see Figure \ref{DL}). Let $\gamma_1$ and $\gamma_2$ be bi-infinite geodesics that have one ray in the equivalence classes of $\omega_1$ and $\omega_2$, and that pass through $o_1$ and $o_2$ respectively. The \textbf{Diestel-Leader graph} $DL(p,q)$ is defined as $$DL(p,q)=\{x_1 x_2\in \mathbb{T}_1\times\mathbb{T}_2\mid \mathfrak{h}_1(x_1)+\mathfrak{h}_2(x_2)=0\}.$$
Vertices $x_1 x_2$ and $y_1 y_2$ are connected by an edge in $DL(p,q)$ if and only if $x_1$ and $y_1$ are connected by an edge in $\mathbb{T}_1$ and $x_2$ and $y_2$ are connected by and edge in $\mathbb{T}_2$.

We say that $x_1 x_2\in DL(p,q)$ is at \textit{level} $n$ if $\mathfrak{h}_2(x_2)=n$, and denote it $\ell(x_1 x_2)=n$.

If $x_1 x_2$ and $y_1 y_2$ are two points in $DL(p,q)$, the following formula describes their distance in $DL(p,q)$ (Proposition 3.1, \cite{bertacchi}):
\begin{equation}\label{bertacchi}\dist(x_1 x_2,y_1 y_2)=\dist_1(x_1,y_1)+\dist_2(x_2,y_2)-|\mathfrak{h}_1(y_1)-\mathfrak{h}_1(x_1)|,
\end{equation}
where $\dist_i(x_i ,y_i)$ is the length of the unique path from $x_i$ to $y_i$ in the tree $\mathbb{T}_i$, and $\mathfrak{h}_i(x)$ is the Busemann function of $x$ (if $x_1 x_2\in DL(p,q)$, then $|\mathfrak{h}_1(x_1)|=~|\mathfrak{h}_2(x_2)|$).

\begin{figure}\caption{Construction of DL(2,3)}\label{DL}
\begin{tikzpicture}[scale = 1,
 		d/.style={circle, fill=black, inner sep=0mm, minimum size = 2pt},
 		e/.style={inner sep=0mm, minimum size = 0pt}]
 	
 	\node (e) at (0,0) [d] {};
 	\node (l) at (-1,-1) [d] {};
 	\node (r) at (1,-1) [d] {};
 	\node (ll) at (-2,-2) [d] {};
 	\node (lr) at (-1,-2) [d] {};
 	\node (rl) at (1,-2) [d] {};
 	\node (rr) at (2,-2) [d] {};
	\node (o1) at (2.2,-1.9) {$o_{1}$};
 	\node (lll) at (-3,-3) [d] {};
 	\node (llr) at (-2.25,-3) [d] {};
 	\node (lrl) at (-1.5,-3) [d] {};
 	\node (lrr) at (-0.75,-3) [d] {};
 	\node (rll) at (0.75,-3) [d] {};
 	\node (rlr) at (1.5,-3) [d] {};  
	\node (omega1) at (-1.4,1.7) {$\omega_{1}$};
        \node (gamma1) at (-0.4,0.7) {\textcolor{blue}{$\gamma_{1}$}};
 	\node (rrl) at (2.25,-3) [d] {};
 	\node (rrr) at (3,-3) [d] {};
 	\node (llls) at (-4,-4) [e] {};
 	\node (llrs) at (-2.5,-4) [e] {};
 	\node (lrls) at (-2,-4) [e] {};
 	\node (lrrs) at (-0.5,-4) [e] {};
 	\node (rlls) at (0.5,-4) [e] {};
 	\node (rlrs) at (2,-4) [e] {};
 	\node (rrls) at (2.5,-4) [e] {};
 	\node (rrrs) at (4,-4) [e] {};
 	
 	\draw	(e) -- (l) -- (ll) -- (lll)
 	(l) -- (lr) -- (lrr)
 	(lr) -- (lrl)
 	(ll) -- (llr);
 	\draw [blue] (e) -- (r) -- (rr) -- (rrr);
 	\draw (r) -- (rl) -- (rll)
 	(rl) -- (rlr)
 	(rr) -- (rrl);
	\draw [-stealth, blue](0,0) -- (-2,2);
 	\draw [dotted] 	(lll) -- (llls)
 	(llr) -- (llrs)
 	(lrl) -- (lrls)
 	(lrr) -- (lrrs)
 	(rll) -- (rlls)
 	(rlr) -- (rlrs)
 	(rrl) -- (rrls);
 	\draw [dotted, blue] (rrr) -- (rrrs);
 	
 	\begin{scope}[shift={(6,-3)}, yscale=-1]
 		\node (ef) at (0,0) [d] {};
 		\node (lf) at (-1,-1) [d] {};
 		\node (rf) at (1,-1) [d] {};
		\node (cf) at (0,-1) [d] {};

		\node (o2) at (-0.8,-1.1) {$o_{2}$};

 		\node (llf) at (-2,-2) [d] {};
 		\node (lrf) at (-1,-2) [d] {};
		\node (lcf) at (-1.5,-2) [d] {};
 		\node (rlf) at (1,-2) [d] {};
 		\node (rrf) at (2,-2) [d] {};
		\node (rcf) at (1.5,-2) [d] {};

		\node (crf) at (0.5,-2) [d] {};
 		\node (ccf) at (0,-2) [d] {};
		\node (clf) at (-0.5,-2) [d] {};

 		\node (lllf) at (-3,-3) [d] {};
		\node (llcf) at (-2.8,-3) [d] {};
 		\node (llrf) at (-2.6,-3) [d] {};
		\node (lclf) at (-2.4,-3) [d] {};
 		\node (lccf) at (-2.2,-3) [d] {};
		\node (lcrf) at (-2,-3) [d] {};
 		\node (lrlf) at (-1.8,-3) [d] {};
 		\node (lrcf) at (-1.6,-3) [d] {};
		\node (lrrf) at (-1.4,-3) [d] {};

		\node (cllf) at (-0.8,-3) [d] {};
		\node (clcf) at (-0.6,-3) [d] {};
 		\node (clrf) at (-0.4,-3) [d] {};
		\node (cclf) at (-0.2,-3) [d] {};
 		\node (cccf) at (0,-3) [d] {};
		\node (ccrf) at (0.2,-3) [d] {};
 		\node (crlf) at (0.4,-3) [d] {};
 		\node (crcf) at (0.6,-3) [d] {};
		\node (crrf) at (0.8,-3) [d] {};

		\node (rrrf) at (3,-3) [d] {};
		\node (rrcf) at (2.8,-3) [d] {};
 		\node (rrlf) at (2.6,-3) [d] {};
		\node (rcrf) at (2.4,-3) [d] {};
 		\node (rccf) at (2.2,-3) [d] {};
		\node (rclf) at (2,-3) [d] {};
 		\node (rlrf) at (1.8,-3) [d] {};
 		\node (rlcf) at (1.6,-3) [d] {};
 		\node (rllf) at (1.4,-3) [d] {};
 		
		\node (lllsf) at (-4,-4) [e] {};
		\node (llcsf) at (-3.6,-4) [e] {};
 		\node (llrsf) at (-3.2,-4) [e] {};
		\node (lclsf) at (-2.8,-4) [e] {};
 		\node (lccsf) at (-2.4,-4) [e] {};
		\node (lcrsf) at (-2.1,-4) [e] {};
 		\node (lrlsf) at (-1.8,-4) [e] {};
 		\node (lrcsf) at (-1.6,-4) [e] {};
		\node (lrrsf) at (-1.4,-4) [e] {};

		\node (cllsf) at (-0.8,-4) [e] {};
		\node (clcsf) at (-0.6,-4) [e] {};
 		\node (clrsf) at (-0.4,-4) [e] {};
		\node (cclsf) at (-0.2,-4) [e] {};
 		\node (cccsf) at (0,-4) [e] {};
		\node (ccrsf) at (0.2,-4) [e] {};
 		\node (crlsf) at (0.4,-4) [e] {};
 		\node (crcsf) at (0.6,-4) [e] {};
		\node (crrsf) at (0.8,-4) [e] {};

		\node (rrrsf) at (4,-4) [e] {};
		\node (rrcsf) at (3.6,-4) [e] {};
 		\node (rrlsf) at (3.2,-4) [e] {};
		\node (rcrsf) at (2.8,-4) [e] {};
 		\node (rccsf) at (2.4,-4) [e] {};
		\node (rclsf) at (2.1,-4) [e] {};
 		\node (rlrsf) at (1.8,-4) [e] {};
 		\node (rlcsf) at (1.6,-4) [e] {};
 		\node (rllsf) at (1.4,-4) [e] {};

 		\node (omega2) at (2,1.7) {$\omega_{2}$};
            \node (gamma2) at (1,0.8) {\textcolor{blue}{$\gamma_{2}$}};
 		\draw [blue]	(ef) -- (lf) -- (llf) -- (lllf);
 	\draw	(lf) -- (lrf) -- (lrrf)
		(lf) -- (lcf) -- (lclf)
		(lcf) -- (lccf)
		(lcf) -- (lcrf)
 		(lrf) -- (lrlf)
		(lrf) -- (lrcf)
 		(llf) -- (llrf)
		(llf) -- (llcf)

		 (ef) -- (cf) -- (clf) -- (cllf)
		(cf) -- (ccf) -- (cclf)
		(cf) -- (crf) -- (crlf)
		(clf) -- (clcf)
		(clf) -- (clrf)
		(ccf) -- (cccf)
		(crf) -- (crcf)
		(ccf) -- (ccrf)
		(crf) -- (crrf)
        
        (ef) -- (rf) -- (rrf) -- (rrrf)
        (rf) -- (rlf) -- (rllf)
		(rf) -- (rcf) -- (rcrf)
		(rcf) -- (rccf)
		(rcf) -- (rclf)
 		(rlf) -- (rlrf)
		(rlf) -- (rlcf)
 		(rrf) -- (rrlf)
		(rrf) -- (rrcf);
 		\draw [-stealth, blue](0,0) -- (2,2);
 		\draw [dotted, blue] 	(lllf) -- (lllsf);
		\draw [dotted] (llcf) -- (llcsf)
 		(llrf) -- (llrsf)
 		(lrlf) -- (lrlsf)
		(lrcf) -- (lrcsf)
 		(lrrf) -- (lrrsf)
		(lclf) -- (lclsf)
		(lccf) -- (lccsf)
 		(lcrf) -- (lcrsf)

		(cllf) -- (cllsf)
		(clcf) -- (clcsf)
 		(clrf) -- (clrsf)
 		(crlf) -- (crlsf)
		(crcf) -- (crcsf)
 		(crrf) -- (crrsf)
		(cclf) -- (cclsf)
		(cccf) -- (cccsf)
 		(ccrf) -- (ccrsf)

		(rllf) -- (rllsf)
		(rlcf) -- (rlcsf)
 		(rlrf) -- (rlrsf)
 		(rrlf) -- (rrlsf)
		(rrcf) -- (rrcsf) 		(rrrf) -- (rrrsf)
		(rclf) -- (rclsf)
		(rccf) -- (rccsf)
 		(rcrf) -- (rcrsf);
 	\end{scope}
 	
 	\draw[thick, decorate, decoration={zigzag, pre length=1.6mm, post length=1.6mm}] (rr) -- (lf);
	\node (a) at (11,-3) [e]{};
	\node (b) at (11,-2) [e]{};
	\node (c) at (11,-1) [e]{};
	\node (aa) at (12,-3) {$\ell=-1$};
	\node (bb) at (12,-2) {$\ell=0$};
	\node (cc) at (12,-1) {$\ell=1$};
	\node (d) at (11,-5) [e]{};
	\node (g) at (11,1) [e]{};
	\draw [-stealth]    (d) -- (g);
	\draw (10.8,-2) -- (11.2,-2);
	\draw (10.8,-1) -- (11.2,-1);
	\draw (10.8,-3) -- (11.2,-3);
 	\end{tikzpicture}
\end{figure}

When evaluating the divergence of the graph, because the graph is vertex-transitive, we can fix the point $x_0$ and use the definition of divergence (\ref{defsimpl}). Use the origin $o_1 o_2$ as $x_0$, so that it plays the same role as the identity in groups. Consider the metrics $\norm{x_1 x_2}=\dist(x_1 x_2,o_1 o_2)$, $\norm{x_1}_1=\dist_1(x_1,o_1)$, $\norm{x_2}_2=~\dist_2(x_2,o_2)$, for $x_1\in\mathbb{T}_1$ and $x_2\in\mathbb{T}_2$. Note that $\norm{x_i}\geq |\mathfrak{h}_{i}(x_i)|$, and therefore also $\norm{x_1 x_2}\geq \norm{x_i}_{i}$, for $i\in\{1,2\}$.

\begin{thm}\label{proposition:DL-graphs} Diestel-Leader graphs $DL(p,q)$ have linear divergence for all \\$p,q\geq~2$.
\end{thm}
\begin{proof}
Our goal is to show that any point $x_1 x_2$ such that $\norm{x_1 x_2}=n$ can be connected to a fixed point $a_1 a_2$ with $\norm{a_1 a_2}=n$ outside $B(o_1 o_2, \frac{n}{3})$ by a path of length linear in $n$.
Let $a_1 a_2$ be the point at level $n$ with $a_1\in\gamma_1$ and $a_2\in\gamma_2$, so $a_1 a_2$ has distance $n$ from $o_1 o_2$.
Let $x_1 x_2\in DL(p,q)$ be such that $\norm{x_1 x_2}=n$. 

Let $b_2$ be a child of $x_2$ with $\norm{b_2}_2=~\norm{x_2}_2+~1$, and let $b_1$ be the parent of $x_1$. 
We start our path from $x_1 x_2$ by reaching $b_1 b_2$. Consider a bi-infinite geodesic $\eta_2$ in $\mathbb{T}_2$ that passes through $x_2$ and $b_2$, and which has one of the rays in the equivalence class of $\omega_2$. In $\mathbb{T}_1$, consider the path from $x_1$ to $a_1$, which lies on the geodesic ray $\eta_1$ in the equivalence class of $\omega_1$ that starts from $x_1$.

Move in $DL(p,q)$ along the path $\eta$ from $x_1 x_2$ that goes up at each step until $\ell=n$ whose projections on $\mathbb{T}_1$ and $\mathbb{T}_2$ are parts of $\eta_1$ and $\eta_2$ respectively. 
We stop when we reach a point at level $n$.

The path $\eta$ avoids a ball of radius $\frac{n}{3}$ around $o_1 o_2$. This is because, if $\ell(x_1 x_2)\geq -\frac{n}{3}$, we avoid such a ball by triangle inequality until $\ell= \frac{n}{3}$, and from $\ell= \frac{n}{3}$ to $\ell= n$ we can use that fact that $|\ell|=|\mathfrak{{h}_2}|$ is a lower bound for the metric in $DL(p,q)$. If $\ell(x_1 x_2)\leq -\frac{n}{3}$, for any point $y_2$ on $\eta_2$, we have that $|y_2|_2\geq |x_2|+1>\frac{n}{3}$, using the fact that the path connecting $y_2$ to $o_2$ passes through $x_2$. Since $|y_1 y_2|\geq |y_2|_2$, we avoid $B(o_1 o_2, \frac{n}{3})$ in this part too. 

In either case, we reach a point $a_1 z_2$ at level $n$ in at most $2n$ steps and without entering the ball $B(o_1 o_2,\frac{n}{3})$ at any time. We are left with connecting $a_1 z_2$ to $a_1 a_2$.

In order to connect $a_1 z_2$ and $a_1 a_2$, we move along a path $\alpha$ whose projection on $\mathbb{T}_2$ is $\alpha_2$, the path from $z_2$ to $a_2$. This path goes down until it reaches $a_2 \curlywedge z_2$, and then up again.
The projection of $\alpha$ on $\mathbb{T}_1$ is a path from $a_1$ that goes down at every step and that passes through another child of $a_1$ than the one on the path from $o_1$ to $a_1$, and then up again along the same path.

In the section from $\ell=n$ to $\ell=\frac{n}{3}$, the lower bound given by $\ell$ implies that we avoid a ball of radius $\frac{n}{3}$. Next, consider a point $y_1 y_2$ on $\alpha$ with $\ell\leq \frac{n}{3}$. In this case, $|y_1 y_2|\geq |y_1|_1=\dist_1(y_1, a_1)+|a_1|_1\geq n$. The path $\alpha$ avoids $B(o_1 o_2,\frac{n}{3})$, and it has length at most $4n$.

We can therefore connect $a_1 a_2$ to $x_1 x_2$ using the concatenation of $\eta$ and $\alpha$, which has length at most $6n$ and which does not enter $B(o_1 o_2, \frac{n}{3})$.
\end{proof}

\subsection{Divergence function of horocyclic products}\label{section:horocyclic products} The above argument can be generalised to show the following result.
\begin{thm}\label{thm:horocyclic}
    Horocyclic products of proper, geodesically complete, Busemann $\delta$-hyperbolic spaces that are uniformly not a quasi-line have linear divergence.
\end{thm}
A metric space is \textbf{geodesically complete} if all geodesics can be extended infinitely. A metric space is \textbf{Busemann} if the distance between any pair of geodesics parametrized by arclength is a convex function. 

Let $X_1$ and $X_2 $ be proper, geodesically complete, Busemann $\delta$-hyperbolic spaces. The horocyclic product $X_1\bowtie X_2$ was defined by Ferragut \cite{ferragut-thesis} as follows.

Let $\omega_1$ be a point in the Gromov boundary of $X_1$, $\omega_2$ be a point in the Gromov boundary of $X_2$, and $o_1$ basepoint in $X_1$, $o_2$ basepoint in $X_2$. Fix $\gamma_1$ and $\gamma_2$ geodesics that represent $\omega_1$ and $\omega_2$ respectively, and that pass through $o_1$ and $o_2$, respectively. Consider Busemann functions, or heights, with respect to $\omega_i$ and $o_i$, i.e. $\mathfrak{h}_i(x)=\lim \sup_{t\to\infty} \dist_i(x, \gamma_i(t)) -t$ (often denoted $\beta_{(\omega_i, o_i)}(x)$).

The horocyclic product is $$X_1\bowtie X_2:=\{x_1 x_2\in X_1\times X_2\mid \mathfrak{h}_1(x_1)+\mathfrak{h}_2(x_2)=0\}.$$
It is equipped with the metric $\dist_{\bowtie}$, defined as follows. Let $N$ be an admissible norm in $\mathbb{R}^2$, that is a norm such that $N(a,b)\geq \frac{a+b}{2}$ for all $a,b\in\mathbb{R}$. All norms on $\mathbb{R}^2$ are equivalent, so there is $C_N\geq 1$ such that $N(a,b)\leq C_N \frac{a+b}{2}$. The metric $\dist_{\bowtie}$ is the length path metric induced by the distance $N(\dist_1 , \dist_2)$. 
For $x, y\in X_1\bowtie X_2$, we have 
\begin{align*}
    \dist_{\bowtie}(x,y)=\inf\{l_N(\gamma)\mid \gamma\ \text{path in}\ X_1\bowtie X_2\ \text{from}\ x\ \text{to}\ y\}.
\end{align*}

A \textbf{vertical geodesic ray} (resp. line) in $X_i$ is a geodesic ray (resp. line) in the equivalence class of $\omega_i$ (resp. which has one of the rays in the equivalence class of $\omega_i$). Vertical geodesics are parametrized by arc-length by their height, so a certain displacement along a vertical geodesic corresponds to the same change of height. If a metric space is Busemann, then the vertical geodesic ray starting at any point is unique.

If $X_1$ and $X_2$ are proper, geodesically complete, Busemann $\delta$-hyperbolic spaces, we have the following formula for the metric (Cor. 4.13 in \cite{ferragut-paper}): there exists $C>0$ such that for $p,q\in X_1 \bowtie X_2$, 
\begin{equation}\label{horocyclicmetric}
    |\dist_{\bowtie}(p,q)-(\dist_1(p_1,q_1)+\dist_2(p_2,q_2)-\Delta \mathfrak{h}(p,q))|\leq C,
    \tag{$\bigstar$}
\end{equation}
    and therefore $\dist_{\bowtie}(p,o_1 o_2)\geq~ \dist_1(p_1,o_1) -~C\geq~ |\mathfrak{h}(p_1)|-~C$.  For a point $x_1 x_2\in X_1 \bowtie X_2$, its level is $\ell(x_1 x_2)=\mathfrak{h}_2(x_2)$. 

For more details about this construction, see \cite{ferragut-paper} and \cite{ferragut-thesis}. Examples of horocyclic products include Baumslag-Solitar groups $BS(1,p)$, when the spaces are $\mathbb{H}^2$ and $\mathbb{T}_{p}$ (the homogeneous tree of degree $p+1$), Diestel-Leader graphs $DL(p,q)$, when the spaces are $\mathbb{T}_p$ and $\mathbb{T}_q$, and Sol, one of the eight Thurston geometries, when both spaces are $\mathbb{H}^2$ (see \cite{woess-survey} for a survey on this). For these three examples, linear divergence is already established (Theorem \ref{proposition:DL-graphs} and because solvable groups have linear divergence), but Theorem \ref{thm:horocyclic} works more generally. 

A metric space is \textbf{uniformly not a quasi-line} if for all $M\geq 0$, there exists $R_M\geq 0$ such that for all geodesics $\gamma$, for all $x\in\gamma$, there exists $y\in B(x,R_M)$ such that $\dist(y, \gamma)\geq M$. 
   \begin{figure}[ht]
  \centering
  \begin{minipage}{0.4\textwidth}
    \centering
    \begin{tikzpicture}[baseline=(current bounding box.center),scale=1]
      \def\R{3}

      \draw[line width=1pt] (0,0) circle (\R);
      \path[use as bounding box] (-\R,-\R) rectangle (\R,\R);

      \draw[line width=1pt] (0,-\R) -- (0,\R);
      \node[right] at (0,-\R+1)       {\small $\gamma_1$};
      \node[above] at (0,\R)          {\small $\omega_1$};

      \fill (0,0) circle (1pt)
        node[above right,inner sep=1pt] {\small $o_1$};
      \node[above left] at (-0.7*\R,0.7*\R) {\small $X_1$};

      \draw[blue, line width=1pt, dotted] (0,0.75*\R) circle (0.25*\R);
      \draw[blue, line width=1pt, dotted] (0,0.25*\R) circle (0.75*\R);
      \draw[line width=1pt, dotted] (0,0.3*\R)  circle (0.7*\R);
      \fill (0,0.5*\R) circle (2pt) node[above right] {\small $a_1$};

      \coordinate (C) at (-4.62813,\R);
      \def\r{4.62813}
      \draw[red, ->, line width=1.2pt]
        plot [variable=\t,domain=-49.43:-18.5,smooth,samples=60]
          ({-4.62813 + \r*cos(\t)}, {3 + \r*sin(\t)});

        \node[red] at (-0.2*\R,0.1*\R) {\small $\eta_1$};
      \draw[red, dotted, line width=1.2pt]
        plot [variable=\t,domain=-18.5:0,smooth,samples=60]
          ({-4.62813 + \r*cos(\t)}, {3 + \r*sin(\t)});

      \fill[red] (-1.6,-0.5) circle (2pt)
        node[below right,inner sep=1pt] {\small $x_1$};
      \coordinate (Z) at 
        ({-4.62813 + \r*cos(-18.5)}, {3 + \r*sin(-18.5)});
      \fill[red] (Z) circle (2pt)
        node[above left,inner sep=1pt] {\small $z_1$};
    \end{tikzpicture}
  \end{minipage}
  \hfill
  \begin{minipage}{0.4\textwidth}
    \centering
    \begin{tikzpicture}[baseline=(current bounding box.center),scale=1]
      \def\R{3}

      \draw[line width=1pt] (0,0) circle (\R);
      \path[use as bounding box] (-\R,-\R) rectangle (\R,\R);

      \draw[line width=1pt] (0,-\R) -- (0,\R);
      \node[below] at (0,-\R)       {\small $\omega_2$};
      \node[right] at (0,\R-1)       {\small $\gamma_2$};

      \fill (0,0) circle (1pt)
        node[above right,inner sep=1pt] {\small $o_2$};
      \node[above right] at (0.7*\R,0.7*\R) {\small $X_2$};

      \draw[blue, line width=1pt, dotted] (0,-0.25*\R) circle (0.75*\R);
      \draw[blue, line width=1pt, dotted] (0,-0.75*\R) circle (0.25*\R);
      \draw[line width=1pt, dotted] (0,-0.7*\R)  circle (0.3*\R);

      \coordinate (P2) at (-0.5,-1.4);
      \fill[red] (P2) circle (2pt) node[below left,inner sep=1pt] {\small $x_2$};

      \coordinate (C) at (-2.9,-\R);
      \def\r{2.9}
      \draw[red, dotted, line width=1.2pt]
        plot[variable=\t,domain=0:33.69,smooth,samples=60]
          ({-2.9 + \r*cos(\t)}, {-3 + \r*sin(\t)});
      \draw[red, ->, line width=1.2pt]
        plot[variable=\t,domain=33.69:75,smooth,samples=60]
          ({-2.9 + \r*cos(\t)}, {-3 + \r*sin(\t)});

    \node[red] at (-0.3*\R,-0.2*\R) {\small $\eta_2$};
      \draw[red, dotted, line width=1.2pt]
        plot[variable=\t,domain=75:92,smooth,samples=60]
          ({-2.9 + \r*cos(\t)}, {-3 + \r*sin(\t)});

      \fill (0,0.5*\R) circle (2pt)
        node[above left] {\small $a_2$};
      \coordinate (P3) at ({-2.9 + \r*cos(75)}, {-3 + \r*sin(75)});
      \fill[red] (P3) circle (2pt) node[below left,inner sep=1pt] {\small $z_2$};
    \end{tikzpicture}

  \end{minipage}
  \caption{The paths $\eta_1$ and $\eta_2$ in the proof of Theorem \ref{thm:horocyclic}. The blue dotted circles are the horospheres with $\mathfrak{h}=n$ and $\mathfrak{h}=-n$.}\label{proof of horocyclic1}
  \centering
\hspace{2cm}
  
  \begin{minipage}{0.4\textwidth}
    \centering
    \begin{tikzpicture}[baseline=(current bounding box.center),scale=1]

  \def\R{3}

  \draw[line width=1pt] (0,0) circle (\R);
   \path[use as bounding box] (-\R,-\R) rectangle (\R,\R);

\draw[blue, line width=1pt, dotted] (0,0.75*\R) circle (0.25*\R);
\draw[line width=1pt, dotted] (0,0.71*\R) circle (0.29*\R);
\draw[line width=1pt, dotted] (0,0.58*\R) circle (-0.42*\R);
  \draw[line width=1pt] (0,-\R) -- (0,\R);
  \node[right] at (0,-\R+1)          {\small $\gamma_1$};
  \node[above] at (0,\R)             {\small $\omega_1$};

  \fill (0,0) circle (1pt) 
    node[above right,inner sep=1pt]  {\small $o_1$};
  \node[above left] at (-0.7*\R,0.7*\R) {\small $X_1$};

  \coordinate (C) at (-6.63,3);
  \def\r{6.63}

  \draw[green!60!black, dotted, line width=1.2pt]
    plot [variable=\t, domain=0:-13, smooth, samples=50]
      ({-6.63 + \r*cos(\t)}, {3 + \r*sin(\t)});

  \draw[green!60!black, <->, line width=1.2pt]
    plot [variable=\t, domain=-13:-22, smooth, samples=50]
      ({-6.63 + \r*cos(\t)}, {3 + \r*sin(\t)});
\node[green!60!black] at (-0.05*\R,0.3*\R) {\small $\alpha_1$};
  \coordinate (pt2) at ({-6.63 + \r*cos(-13)}, {3 + \r*sin(-13)});
  \fill[green!60!black] (pt2) circle (2pt)
    node[above left, inner sep=1pt] {\small $a_1''$};

    \coordinate (pt2) at ({-6.63 + \r*cos(-15)}, {3 + \r*sin(-15)});
  \fill[green!60!black] (pt2) circle (2pt)
    node[left, inner sep=1pt] {\small $a_1'$};

\end{tikzpicture}

\end{minipage}
  \hfill
  \begin{minipage}{0.4\textwidth}
    \centering
    \begin{tikzpicture}[baseline=(current bounding box.center),scale=1]

  \def\R{3}

  \draw[line width=1pt] (0,0) circle (\R);
  \path[use as bounding box] (-\R,-\R) rectangle (\R,\R);
   \draw[line width=1pt, dotted] (0,-0.29*\R) circle (0.71*\R);
\draw[blue, line width=1pt, dotted] (0,-0.25*\R) circle (0.75*\R);
\draw[line width=1pt, dotted] (0,-0.42*\R) circle (0.58*\R);

  \draw[line width=1pt] (0,-\R) -- (0,\R);
  \node[below] at (0,-\R) {\small $\omega_2$};
  \node[right] at (0,\R-1)  {\small $\gamma_2$};

  \fill (0,0) circle (1pt) node[above right,inner sep=1pt] {\small $o_2$};
  \node[above right] at (0.7*\R,0.7*\R) {\small $X_2$};

  \coordinate (A) at ($(0,-\R)!0.75!(0,\R)$);
  \draw[line width=1pt]
    ($(A)+(-0.1,0)$) -- ($(A)+(0.1,0)$)
    node[above right,inner sep=1pt, green!60!black] {\small $a_2$};
\fill[green!60!black] (A) circle (2pt);
  \coordinate (C) at (-1.45*\R,1.25*\R);
  \def\r{1.6326*\R}
 
  \draw[green!60!black, ->, line width=1.2pt]
    plot [variable=\t, domain=-62:-27, smooth, samples=60]
      ({-1.45*\R + \r*cos(\t)}, {1.25*\R + \r*sin(\t)});
\node[green!60!black] at (-0.4*\R,0.1*\R) {\small $\alpha_2$};
\coordinate (P3) at ({-0.7*\R}, {-0.2*\R});
      \fill[green!60!black] (P3) circle (2pt) node[below right,inner sep=1pt] {\small $z'_2$};
      
\end{tikzpicture}

  \end{minipage}
   \caption{The paths $\alpha_1$ and $\alpha_2$ in the proof of Theorem \ref{thm:horocyclic}: $\alpha_1$ starts at $a'_1$, goes down and then up again up to $a''_1$; $\alpha_2$ goes from $z'_2$ to $a_2$.}\label{proof of horocyclic2}
\end{figure}

\begin{proof}[Theorem \ref{thm:horocyclic}]
The goal is to connect any point $x_1 x_2$ on a sphere of radius $n$ around $o_1 o_2$ to a fixed point $a_1 a_2$, also at distance $n$ from $o_1 o_2$, using a path that avoids a ball of radius $\frac{n}{3}-C$. The argument is depicted in Figures \ref{proof of horocyclic1} and \ref{proof of horocyclic2} using the Poincaré disc model for the hyperbolic plane.

Let $a_1 a_2$ be the point of $X_1\bowtie X_2$ with $a_1\in \gamma_1$, $a_2\in \gamma_2$ such that $\ell(a_1 a_2)=~n$. 
Consider a point $x_1 x_2$ with $\dist_{\bowtie}(x_1 x_2)=n$. Up to some initial short path of length at most $3\delta +R_{3\delta}+C$, we may assume that $\dist_2(x_2,\gamma_2)\geq~3\delta$. This is because of our assumption that $X_2$ is uniformly not a quasi-line.
    
    We can define a path similar to $\eta$ in the case of Diestel-Leader graphs in order to reach level $n$. 
    Let $\eta_1$ be the vertical geodesic through $x_1$ in $X_1$. Let $\eta_2$ be the vertical geodesic through $x_2$ in $X_2$. 

    Like for $DL$-graphs, move along $\eta_1$ in $X_1$ and $\eta_2$ in $X_2$ until we reach level $n$. We reach the point $z_1 z_2$, for some $z_1\in X_1$, $z_2\in X_2$.
    
    Since $\dist_2(x_2,\gamma_2)\geq ~ 3\delta$, and $\gamma_2$ and $\eta_2$ are vertical geodesics, by the Busemann condition, $\dist_2(\eta_2(t), \gamma_2(t))\geq 3\delta $ for $t\leq ~-\mathfrak{h}_2(x_2)$. 
    
    If $\ell(x_1 x_2)\geq \frac{n}{3}$, for any $y=y_1 y_2\in~\eta$ we can use the lower bound for $\dist_{\bowtie} (y_1 y_2 , o_1 o_2)$ given by $|\mathfrak{h}_2(y_2)|-C$ to conclude that $\eta$ avoids $B(o_1 o_2,\frac{n}{3}-C)$. 
    
    If $-\frac{n}{3}\leq \ell(x_1 x_2)\leq\frac{n}{3}$, then we avoid $B(o_1 o_2,\frac{n}{3}-C)$ up to level $\frac{n}{3}$ by triangle inequality, as we are moving along vertical geodesics, and from level $\frac{n}{3}$ using the lower bound given by the height.
    
    If $\ell(x_1 x_2)\leq -\frac{n}{3}$, for any $y=y_1 y_2\in~\eta$, consider the quadrilateral with vertices $y_2, x_2, q_2, o_2$, where $q_2$ is the point on $\gamma_2$ at the same height as $x_2$. The quadrilateral is $2\delta$-~slim, hence $\dist_2(y_2,o_2)\geq \frac{n}{3}-8\delta$. This is because the point on $\gamma_2 $ at level $-\frac{n}{3}+6\delta$ has distance $\geq 3\delta$ from $[x_2,y _2]$ and from $[q_2,x_2]$ (by the Busemann condition, since $\gamma_2$ and $\eta_2$ are both vertical geodesics, and since the height of a point in $[q_2,x_2]$ does not exceed $-\frac{n}{3}+2\delta$), so its distance from $[y_2,o_2]$ has to be $\leq 2\delta$. 

     Hence $\eta$ connects $x_1 x_2$ to $z_1 z_2$ avoiding $B(o_1 o_2,\frac{n}{3}-C)$, and has length at most $2n+3\delta ~+R_{3\delta}+~2C$, since we move along vertical geodesics and using (\ref{horocyclicmetric}).

    Next we reach $a_1 z_2$. We do it using the geodesic segment $[z_1,a_1]$ in $X_1$, and moving along $\eta_2$ in $X_2$. This path has length at most $\dist_1(z_1,a_1)+~C~\leq 4n+~2C$, and it avoids $B(o_1, n-2\delta)$. To see this, we may consider the quadrilateral with vertices $z_1$, $a_1$, and with two vertices on $\eta_1$, $\gamma_1$ at height $2n$, which is $2\delta$-slim. 

    From $a_1 z_2$, we get to $a'_1 z'_2$, where $a'_1$ is a point such that $\dist(a'_1,\gamma_1)\geq 3\delta$ and $\dist(a_1, a'_1)\leq R_{3\delta}$, whose existence is guaranteed by $X_1$ being uniformly not a quasi-line (and $z'_2$ is the point at the same height on $\eta_2$). This short path of length at most $R_{3\delta}+C$ allows us to find a new geodesic that diverges from $\gamma_1$.
    
     We now define a path similar to $\alpha$ in the case of Diestel-Leader graphs. 
    Let $\alpha_1$ be the vertical geodesic through $a'_1$ in $X_1$. Let $\alpha_2$ be the geodesic segment $[z'_2, a_2]$ in $X_2$. The path $\alpha$ starts at $a'_1 z'_2$, it has projection $\alpha_2$ on $X_2$, and the projection on $X_1$ is contained in $\alpha_1$. It reaches a point $a''_1 a_2$, where $a''_1$ is the point on $\alpha_1$ at height $n$. 

This path has length at most $4n+3C+R_{3\delta}$ and it avoids $B(o_1 o_2, \frac{n}{3}-C)$. Indeed, for any $y_1 y_2\in \alpha$ with $\ell(y_1 y_2)\geq\frac{n}{3}$, we can use the lower bound given by the height. If $\ell(y_1 y_2)<\frac{n}{3}$, then $\dist_1(y_1,o_1)\geq n-8\delta$ by considering the quadrilateral with vertices $y_1, a''_1, a_1, o_1$, which is $2\delta$-slim.
  Finally, a short path of length at most $2R_{3\delta}+C$ allows us to reach $a_1 a_2$. 

Concatenating the paths obtained, we can connect $x_1 x_2$ to $a_1 a_2$ with a path of length $\leq 10n+3\delta+ 9C+5R_{3\delta}$ avoiding $B(o_1 o_2, \frac{n}{3}-C)$.

Note that in this case fixing the third point in the definition of divergence is in principle not possible, as the isometry group need not act transitively. However in this case, if we have a third point $w_1 w_2$, we can just replace $\gamma_i$ with vertical geodesics through $w_i$, and produce a similar path.

\end{proof}

\section{Divergence function of Houghton groups}\label{section:houghton groups}
For $m\geq 1$, we define $R_m$ to be $\{1,\dots,m\}\times\mathbb{N}^+$, the graph with $m$ pairwise disjoint rays and where the edges connect vertices $(i,k)$ and $(i,k+1)$, for $i\in~\{1,\dots, m\}$ and $k\geq 1$. The \textbf{Houghton group} $\mathcal{H}_m$, introduced by Houghton in \cite{Houghton},  is the group of permutations of $R_m$ that are eventually translations for each of the rays $\{i\}\times \mathbb{N}^+$.

For example, $\mathcal{H}_1$ is $\FSym(\N)$, the finitary symmetric group on $R_1$, which is not finitely generated. Next, $\mathcal{H}_2$ is $\FSym(\Z)\rtimes\Z$. The derived subgroup of $\mathcal{H}_m$ is either $\FAlt(R_m)$ or $\FSym(R_m)$ (see \cite{commH}), depending on $m$, so these groups are far from being solvable. In this section we show the following fact, using a similar argument as the one for wreath products.
\begin{thm}\label{houghton} The Houghton groups $\mathcal{H}_m$ have linear divergence for all $m\geq 2$.
\end{thm}

The Houghton group $\mathcal{H}_2$, also known as the \textit{lampshuffler group}, is the group $\FSym(\mathbb{Z})\rtimes \mathbb{Z}$, and can also be defined as the subgroup of the group of bijections of $\{\dots, -1,1,2,\dots\}=\mathbb{Z}\setminus\{0\}$ generated by a translation $t$ and a transposition $a=(-1\ 1)$. 
The group $\mathcal{H}_2$ has presentation $$\langle a,t\mid a^2=1, (a a^{t})^{3}=1,[a,a^{t^{k}}]=1,k\geq 2\rangle.$$
We can think of an element of $\mathcal{H}_2$ as a cursor that starts between $-1$ and $1$, moves on $\mathbb{Z}\setminus\{0\}$ to the right for each occurrence of $t$ and to the left for each occurrence of $t^{-1}$, together with transpositions of two consecutive elements, corresponding to the occurrences of $a$.

The cursor is always between two integers; define its \textbf{position} to be the integer to its right, so that the initial position of the cursor is $1$.

We consider now Houghton groups $\mathcal{H}_m$ with $m\geq 3$. Define $g_i$ to be the element of $\mathcal{H}_m$ that translates the union of rays $i$ and $i+1$. Then, the commutator $[g_1,g_2]$ is the transposition of $(1,1)$ and $(2,1)$. It follows that in this case $\mathcal{H}_m$ is generated by $S=\{g_1,\dots, g_{m-1}\}$. Consider the Cayley graph of $\mathcal{H}_m$ with respect to the generating set $S\cup S^{-1}$. 

\begin{proof}
There are two parts to this proof; in the first part we show the result for $\mathcal{H}_2$, and in the second part we focus on $\mathcal{H}_m$ with $m\geq 3$.
We use $S=\{t,a\}$ as generating set for $\mathcal{H}_2$ and consider the Cayley graph of $\mathcal{H}_2$ with respect to $S\cup S^{-1}=\{a,t,t^{-1}\}$. The goal of the proof is to produce a path with length linear in $n$ and avoids $B(1,\frac{n}{2})$, that connects any element $g$ with $\norm{g}=n$ to the fixed element $t^{n-1}a$, which also has length $n$.

Let $g\in \mathcal{H}_2$ be such that $\norm{g}=n$. 
The following is a bi-infinite geodesic through $g$: $$\dots, gt^{-1},g,gt,\dots.$$ 
By Lemma \ref{Lemma:rays}, one of its rays does not intersect $B(1,\frac{n}{2})$. Move by $3n$ steps along such ray, that is multiply by $t^{k}$, where $k=3n$ or $k=-3n$. Then apply a transposition $a$. Since $\norm{g}=n$, the position of the cursor is in $[-n,n+1]$ after applying $g$ to $\mathbb{Z}\setminus\{0\}$. It will therefore be either in $[2n+1, 4n+1]$ or $[-4n,-2n]$ after applying $t^{k}$. Let $\ell$ be the sum of the exponents of the occurrences of $t$ in any word representing $g$. The transposition we applied after $t^{k}$ is at position $\ell+k+1$, if it is in $[2n+1, 4n+1]$, or at position $\ell+k$, if it is in $[-4n,-2n]$.

Next, we undo all the transformations given by $g$ and apply a transposition at $n$. That is, we first apply $t^{-k}$ to reach the position the cursor is at after applying $g$, and then $g^{-1}$, to undo the transformations given by $g$. The element we get to is $gt^{k}at^{-k}g^{-1}=t^{k+\ell}at^{-k-\ell}$.

We proceed by multiplying by $t^{n-1}a$ in order to apply the transposition $(n-1\ n)$. 
This will be disjoint from the one at $k+\ell$, since $|k+\ell|\geq 2n$, which ensures that they don't cancel out.

Next, we undo the transposition at $k+\ell$. That is, apply $t^{-(n-1)+k+\ell}at^{-k-\ell+n-1}$. 
The element we obtain has cursor at $n$ and a transposition $(n-1\ n)$, that is we are left with the element $t^{n-1}a$, which has length $n$.

More precisely, the path from $g$ to $t^{n-1}a$ consists in multiplying $g$ by $t^{k}at^{-k}g^{-1}t^{n-1}at^{-(n-1)+k+\ell}at^{-k-\ell}t^{n-1}$.

At any step of the path, we are outside $B(1,\frac{n}{2})$, since we are either on the chosen ray that avoids this ball, or we are at an element that has a transposition at $q$, with $|q|\geq n$.
If an element has a transposition at position $q$, then its metric is at least $|q|$. 
The path we used has length at most $18n$.

We can therefore connect any element of length $n$ to the fixed element $t^{n-1}a$ in linear time outside $B(1, \frac{n}{2})$. Hence any two elements of length $n$ can be connected by a path of linear length outside $B(1, \frac{n}{2})$, and divergence is linear.

We consider now $\mathcal{H}_m$ with $m\geq 3$.
We will adapt the argument for linearity of divergence of $\mathcal{H}_2$ to this case, by replacing $a$ with $[g_1,g_2]$, which has length $4$, and we take $g_1^{n-4}[g_1,g_2]$ as fixed element of length $n$. 

Take $g$ of length $n$. Consider the bi-infinite geodesic $\dots,gg_1^{-1},g,gg_1,\dots$ through $g$ in the Cayley graph. Applying Lemma \ref{Lemma:rays}, one of the infinite rays of this geodesic does not enter $B(1,\frac{n}{2})$. Let $R$ be an infinite ray starting at $g$ that does not enter $B(1,\frac{n}{2})$. Move along $R$ by $3n$ steps and reach $gg_1 ^{k}$, where $k=\pm 3n$. That corresponds to translating the union of rays $1$ and $2$ of $R_m$ either to the left or to the right, depending on which ray is chosen.

Next, we apply the transposition $[g_1,g_2]$, and we multiply by $g_1^{-k}$. Since $\norm{g}=n$, the part of $R_m$ where $g$ does not act only by translation is inside $\{(p,q)\mid q\leq n\}$. The element we obtain has the same configuration of $g$ on $R_m$, with an extra transposition $((1,y)(1,y+1))$ or $((2,y)(2,y+1))$ (depending on the sign of $k$) with $y\geq 2n$. The new transposition is therefore disjoint from the part of $R_m$ where $g$ does not act only by translation.

We then multiply our element by $g^{-1}$, in order to undo the action of $g$ on $R_m$. The element we reach is $gg_1^{k}[g_1,g_2]g_1^{-k}g^{-1} $, which is $g_1^{3n+\ell_1}[g_1,g_2]g_1^{-3n-\ell_1} $ if $k=3n$, and $g_1^{-3n+\ell_2 -\ell_1}[g_1,g_2]g_1^{3n-\ell_2+\ell_1} $ if $k=-3n$. Here $\ell_i$ is the sum of the exponents of $g_i$ in a (any) word representing $g$, i.e. the amount by which the union of rays $i$ and $i+1$ is translated by $g$. If $k=3n$, then $g_1^{k}[g_1,g_2]g_1^{-k}$ is a transposition on the first ray, so any generator $g_i$ with $i\neq 1$ commutes with it, and conjugation by $g_1$ contributes positively to the exponent $k$. If $k=-3n$, then the transposition $g_1^{k}[g_1,g_2]g_1^{-k}$ is on the second ray, and any generator $g_i$ with $i\neq 1,2$ commutes with it; conjugation by $g_1$ contributes negatively to the exponent and by $g_2$ positively.

In either case, since $3n-\ell_1\geq 2n$ and $3n-\ell_2+\ell_1\geq 2n$, we get to an element which consists of a transposition $((1,y)(1,y+1))$ or $((2,y)(2,y+1))$ for some $y\geq 2n$. 
At this point, we multiply by $g_1^{n-4}[g_1,g_2]g_1^{-n+4}$ in order to apply a transposition $((1,n-4)(1,n-3))$. This commutes with the other one we had already, so we can now multiply by $g_1^{3n+\ell_1}[g_1,g_2]g_1^{-3n-\ell_1}$ or $g_1^{-3n+\ell_2 -\ell_1}[g_1,g_2]g_1^{3n-\ell_2+\ell_1} $ (depending on the sign of $k$), and we will be left with element $g_1^{n-4}[g_1,g_2]g_1^{-n+4}$. Finally, multiply by $g_1^{n-4}$ to get to our desired fixed element $g_1^{n-4}[g_1,g_2]$.

This path has length at most $18n$, and it avoids $B(1,\frac{n}{2})$ since at each step we are either on $R$ or the element has a transposition at $((1,y)(1,y+1))$ or $((2,y)(2,y+1))$ with $y\geq n-4$, so the metric is at least $n-4$. Thus $\mathcal{H}_m$ has linear divergence for all $m\geq 3$.

This concludes the proof of Theorem \ref{houghton} and, therefore, $\mathcal{H}_m$ has linear divergence for all $m\geq 2$.
\end{proof}
\section{Divergence function of halo products}\label{section:halo products}
Recently Genevois and Tessera introduced a class of groups that they called halo products \cite{gen-tes}, which includes wreath products and Houghton group $\mathcal{H}_2$, but also all the lampshuffler groups, i.e. groups of the form $\FSym(F)\rtimes H$. Here $\FSym(H)$ is the group of finitely supported permutations of $H$, and the action of $H$ on $\FSym(H)$ is by left-multiplication. 

\begin{definition}
    Let $H$ be a group. A halo of groups $\mathcal{L}$ over $H$ is the data, for every
subset $S \subseteq  H$, of a group $L(S)$ such that:
\begin{enumerate}
    \item for all $R, S \subseteq H$, if $R \subseteq S$ then $L(R) \leq L(S)$;
    \item $L(\emptyset) = \{1\}$ and $L(H) = \langle L(S)\mid S \subseteq H\ \text{finite}\rangle$;
    \item for all $R, S \subseteq H$, $L(R) \cap L(S) = L(R \cap S)$.
\end{enumerate}
Given a morphism $\alpha : H \to \operatorname{Aut}(L(H))$ satisfying $\alpha(h)(L(S)) = L(hS)$ for all $S \subseteq G,
h \in H$, the halo product $\mathcal{L}_\alpha H$ is the
semidirect product $L(H) \rtimes_{\alpha} H$.
\end{definition}
A halo of groups $\mathcal{L}$ over a group $H$ is large-scale commutative if there exists a constant $D \geq 0$ such that, for all subsets $R, S \subseteq X$ at distance at least $D$, the subgroups $L(R)$ and $L(S)$ commute.
\begin{thm}\label{thm:halo products}
    Let $H$ be a finitely generated group and $G=L(H)\rtimes_{\alpha} H$ a halo product. If $G$ is finitely generated, and the halo is large-scale commutative, then $G$ has linear divergence.
\end{thm}
\begin{proof}
    We can run a very similar argument as in the case of wreath products, exploiting the following facts:
    \begin{itemize}
        \item $H$ is undistorted in $G$, and we can therefore use the $H$-metric of the $H$-~projection of an element as lower bound for its $G$-metric;
        \item we have commutativity of elements in $L(H)$ whose support is far enough.
    \end{itemize}
    Let $T_H$ be a finite generating set for $H$, and $T$ a finite subset of $L(H)$ such that $G=\langle T_H\sqcup T\rangle$.
Let $K$ be the support of $T$, that is the minimal subset of $H$ such that $T$ is contained in its halo. The extra complication that we encounter in this case with respect to the wreath product case is that $K$ is not $\{1_H\}$. It is however still a finite set, so we can go around this problem by just taking longer (still linear) paths. In order to do that, we can use Lemma \ref{Lemma:rays} which allows us to get far enough from $1_H$ that multiplying by some $t\in T$ does not interact with the element we started with.
\end{proof}
\section{Divergence function of Baumslag-Solitar groups}\label{section:baumslag-solitar groups}
In this section we study the divergence of Baumslag-Solitar groups. Recall that, for integers $p,q\geq 1$, the Baumslag-Solitar group $BS(p,q)$ is defined by the presentation $\langle a,t\mid ta^{p}t^{-1}=a^{q}\rangle$. We show the following proposition.
\begin{proposition}\label{proposition:bsgroups}
    Baumslag-Solitar groups $BS(p,q)$ have linear divergence for all $p,q\geq 1$.
\end{proposition}
It is a result from \cite{DMS10} and \cite{DS05} that finitely-generated one-ended solvable groups have linear divergence, so $BS(1,p)$ has linear divergence for any $p\geq 1$. By \cite{whyte2004large}, the groups $BS(p,p)$ with $p\geq 2$ contain $\Z\times F_p$ of finite index, having thus linear divergence. We are only left with the case $BS(p,q)$, with $1<p< q$. 
\begin{proof} It is shown in \cite{whyte2004large} that all $BS(p,q)$ with $1<p<q$ are quasi-isometric, so it is sufficient to show that $BS(2,4)$ has linear divergence in order to conclude that all the Baumslag-Solitar groups have linear divergence.

Let $G$ be $BS(2,4)=\langle a,t\mid ta^{2}t^{-1}=a^{4}\rangle$ and consider its Cayley graph with respect to the generating set $\{a^{\pm 1},t^{\pm 1}\}$. Take $g\in G$ of length $2n$, and let $w$ be a shortest word representing $g$. We will connect any such $g$ to $a^{2^{2n+1}}=t^{2n}a^2 t^{-2n}$, a fixed element of length $4n+2$, in linear time and avoiding the ball $B(1,n)$.

Define the $t$-level $\ell$ and the $a$-level $k$ of $g$ as follows. Let $a^{i_1}t^{j_1}\dots a^{i_h}t^{j_h}$, where only $i_1$ and $j_h$ can possibly be $0$, be a word representing $g$. The $t$-level $\ell$ of this word is the sum of the $t$-exponents $j_1+\dots +j_h$, and $a$-level $k$ of this word is $i_1+i_2 2^{j_1}+i_3 2^{j_{1}+j_{2}}+\dots +i_h 2^{j_{1}+\dots +j_{h-1} }$. Both $\ell$ and $k$ are invariant under applying the group relation to the word, so they are well-defined for group elements.  In other words, the $t$-level of an element is the $t$-coordinate of the projection onto the $t$-axis of the Cayley graph, and similarly for $a$.

We assume that $k\geq 0$. Let $\Gamma_0$ be the sheet of the Cayley graph that contains the geodesic $\{t^{j}\mid j\in\Z\}$.

We use the bi-infinite geodesic $\{\dots,ga^{-1}t^{2},ga^{-1}t,ga^{-1},g,ga,gat,gat^{2},\dots\}$ through $g$. Then, by Lemma \ref{Lemma:rays}, one of the rays $R_1=\{ga t^{m}\mid m\in \N\}$ and $R_2=\{ga^{-1}t^{m}\mid m\in \N\}$ does not enter the ball $B(1,n)$. Let $R$ be such a ray, $R=\{ga^{\pm 1} t^{m}\mid m\in \N\}$. 

Let $\mathfrak{p}_1$ be the path starting from $g$ and consisting of $2n-\ell+1$ steps along $R$, i.e. whose vertices are $g, ga^{\pm 1}, ga^{\pm 1}t,\dots,ga^{\pm 1}t^{2n-\ell}$. It is on $R$, so it is outside $B(1,n)$. Let $\mathfrak{p}_2$ be $ga^{\pm 1}t^{2n-\ell}, ga^{\pm 1}t^{2n-\ell}a, ga^{\pm 1}t^{2n-\ell}a^{2}$.

Let $\mathfrak{p}_3$ be the path starting at $ga^{\pm 1}t^{2n-\ell}a^{2}$ and obtained multiplying by $t^{-1}$ for $2n-\ell$ times. The $a$-level of $\mathfrak{p}_3$ is $k\pm 2^\ell +2^{2n+1}\geq k+2^{2n+1}-2^{\ell}\geq k+2^{2n+1}-2^{2n}\geq 2^{2n}$.
 
Let $\mathfrak{p}_4$ be the $1$-edge path $ga^{\pm 1}t^{2n-\ell}a^{2}t^{-(2n-\ell)}, ga^{\pm 1}t^{2n-\ell}a^{2}t^{-(2n-\ell)}a^{\mp1}=ga^{2^{2n-\ell+1}}$, where the sign of the exponent of $a$ is the opposite sign than the one in $R$. The path $\mathfrak{p}_1 \cup \mathfrak{p}_2 \cup \mathfrak{p}_3 \cup \mathfrak{p}_4$ is a path from $g$ to $ga^{2^{2n-\ell+1}}$ that increases the $a$-level of $g$ in a number of steps that is linear in $n$.

Let $\mathfrak{p}_5$ be the path starting at the end of $\mathfrak{p}_4$ that consists of multiplying by the generators of a shortest word representing $g^{-1}$, therefore of length $2n$. We have that $ta^{2^{m}}t^{-1}=a^{2^{m+1}}$ and so $ga^{2^{2n-\ell+1}}g^{-1}=a^{2^{2n+1}}$. Therefore $\mathfrak{p}_5$ connects the end of $\mathfrak{p}_4$ to $a^{2^{2n+1}}$. The $a$-level of $ga^{\pm 1}t^{2n-\ell}a^{2}t^{-(2n-\ell)}a^{\mp 1}$ is $k+2^{2n+1}$ and of $a^{2^{2n+1}}$ is $2^{2n+1}$. Hence, the $a$-level of $\mathfrak{p}_5$ is always at least $2^{2n+1}-2^{2n-1}>2^{2n}$, since a shortest path that represents an element of length $2n$ has $a$-levels in $[-2^{2n-1},2^{2n-1}]$. 

The paths $\mathfrak{p}_3$ and $\mathfrak{p}_5$ avoid $B(1,2n)$, because their $a$-level is at least $2^{2n}$. This is because if an element $h$ has $\norm{h}\leq m$, then its $a$-level is at most $2^{m-1}$.

We have that $\mathfrak{p}_1 \cup \mathfrak{p}_2 \cup \mathfrak{p}_3 \cup \mathfrak{p}_4\cup \mathfrak{p}_5$ is a path of length $6n-2\ell +4\leq 10n +4$ from $g$ to $a^{2^{2n+1}}$ that avoids the ball $B(1, n)$.

Now consider elements with $k<0$. By symmetry, we can apply the same procedure and connect them to $a^{-2^{2n+1}}$. The path given by multiplying by $t^{2n}a^{4}t^{-2n}$ connects $a^{-2^{2n+1}}$ to $a^{2^{2n+1}}$ in $4n+4$ steps avoiding $B(1,2n)$.
Thus any element of length $2n$ can be connected to $a^{2^{2n+1}}=t^{2n}a^2 t^{-2n}$ in time linear in $n$ outside $B(1,n)$, yielding linearity of divergence.
\end{proof}

\bibliographystyle{spmpsci}      
\bibliography{refs}

@article{DMS10,
    AUTHOR = {Dru\c{t}u, Cornelia and Mozes, Shahar and Sapir, Mark},
     TITLE = {Divergence in lattices in semisimple {L}ie groups and graphs
              of groups},
   JOURNAL = {Trans. Amer. Math. Soc.},
  FJOURNAL = {Transactions of the American Mathematical Society},
    VOLUME = {362},
      YEAR = {2010},
    NUMBER = {5},
     PAGES = {2451--2505},
      ISSN = {0002-9947,1088-6850},
   MRCLASS = {20F67 (20F65)},

MRREVIEWER = {Olivier\ Guichard},
       DOI = {10.1090/S0002-9947-09-04882-X},
       URL = {https://doi.org/10.1090/S0002-9947-09-04882-X},
}

@article{BD14,
    AUTHOR = {Behrstock, Jason and Dru\c{t}u, Cornelia},
     TITLE = {Divergence, thick groups, and short conjugators},
   JOURNAL = {Illinois J. Math.},
  FJOURNAL = {Illinois Journal of Mathematics},
    VOLUME = {58},
      YEAR = {2014},
    NUMBER = {4},
     PAGES = {939--980},
      ISSN = {0019-2082,1945-6581},
   MRCLASS = {20F65 (20F10 20F67 20F69 53C23 57M07 57N10)},

MRREVIEWER = {Stephan\ Rosebrock},
       URL = {http://projecteuclid.org/euclid.ijm/1446819294},
}

@article{DS05,
    AUTHOR = {Dru\c{t}u, Cornelia and Sapir, Mark},
     TITLE = {Tree-graded spaces and asymptotic cones of groups},
      NOTE = {With an appendix by Denis Osin and Mark Sapir},
   JOURNAL = {Topology},
  FJOURNAL = {Topology. An International Journal of Mathematics},
    VOLUME = {44},
      YEAR = {2005},
    NUMBER = {5},
     PAGES = {959--1058},
      ISSN = {0040-9383},
   MRCLASS = {20F67 (57M07)},

MRREVIEWER = {Ilya\ Kapovich},
       DOI = {10.1016/j.top.2005.03.003},
       URL = {https://doi.org/10.1016/j.top.2005.03.003},
}

@article{Ger94b,
 	title={Quadratic divergence of geodesics in {CAT}(0)-spaces},
  	author={Gersten, Stephen M},
  	journal={Geometric \& Functional Analysis},
  	volume={4},
  	number={1},
  	pages={37--51},
  	year={1994},
  	publisher={Springer}
}

@article{Gr3,
 	title={Geometric group theory, {V}ol. 2: {A}symptotic invariants of infinite groups},
  	author={Gromov, Mikha{\i}l},
  	journal={Bull. Amer. Math. Soc},
  	volume={33},
  	number={0273},
  	pages={0979},
  	year={1996}
}

@article{duchin2010divergence,
    AUTHOR = {Duchin, Moon and Rafi, Kasra},
     TITLE = {Divergence of geodesics in {T}eichm\"{u}ller space and the
              mapping class group},
   JOURNAL = {Geom. Funct. Anal.},
  FJOURNAL = {Geometric and Functional Analysis},
    VOLUME = {19},
      YEAR = {2009},
    NUMBER = {3},
     PAGES = {722--742},
      ISSN = {1016-443X,1420-8970},
   MRCLASS = {30F60 (20F65)},

MRREVIEWER = {Lee-Peng\ Teo},
       DOI = {10.1007/s00039-009-0017-3},
       URL = {https://doi.org/10.1007/s00039-009-0017-3},
}

@article{OOS05,
    AUTHOR = {Ol\cprime shanskii, Alexander Yu. and Osin, Denis V. and
              Sapir, Mark V.},
     TITLE = {Lacunary hyperbolic groups},
      NOTE = {With an appendix by Michael Kapovich and Bruce Kleiner},
   JOURNAL = {Geom. Topol.},
  FJOURNAL = {Geometry \& Topology},
    VOLUME = {13},
      YEAR = {2009},
    NUMBER = {4},
     PAGES = {2051--2140},
      ISSN = {1465-3060,1364-0380},
   MRCLASS = {20F67 (20F65 20F69)},

MRREVIEWER = {Goulnara\ N.\ Arzhantseva},
       DOI = {10.2140/gt.2009.13.2051},
       URL = {https://doi.org/10.2140/gt.2009.13.2051},
}

@article{floyd,
    AUTHOR = {Floyd, William J.},
     TITLE = {Group completions and limit sets of {K}leinian groups},
   JOURNAL = {Invent. Math.},
  FJOURNAL = {Inventiones Mathematicae},
    VOLUME = {57},
      YEAR = {1980},
    NUMBER = {3},
     PAGES = {205--218},
      ISSN = {0020-9910,1432-1297},
   MRCLASS = {57M15 (22E40 30F40 51M20)},

MRREVIEWER = {I.\ Kra},
       DOI = {10.1007/BF01418926},
       URL = {https://doi.org/10.1007/BF01418926},
}

@article{woess,
    AUTHOR = {Woess, Wolfgang},
     TITLE = {Lamplighters, {D}iestel-{L}eader graphs, random walks, and
              harmonic functions},
   JOURNAL = {Combin. Probab. Comput.},
  FJOURNAL = {Combinatorics, Probability and Computing},
    VOLUME = {14},
      YEAR = {2005},
    NUMBER = {3},
     PAGES = {415--433},
      ISSN = {0963-5483,1469-2163},
   MRCLASS = {60C05 (05C25 60G50 60J45)},

MRREVIEWER = {Donald\ I.\ Cartwright},
       DOI = {10.1017/S0963548304006443},
       URL = {https://doi.org/10.1017/S0963548304006443},
}

@article{eskin2006quasiisometries,
    AUTHOR = {Eskin, Alex and Fisher, David and Whyte, Kevin},
     TITLE = {Coarse differentiation of quasi-isometries {I}: {S}paces not
              quasi-isometric to {C}ayley graphs},
   JOURNAL = {Ann. of Math. (2)},
  FJOURNAL = {Annals of Mathematics. Second Series},
    VOLUME = {176},
      YEAR = {2012},
    NUMBER = {1},
     PAGES = {221--260},
      ISSN = {0003-486X,1939-8980},
   MRCLASS = {22E25 (05C05 05C76 20F65)},

MRREVIEWER = {B.\ Sury},
       DOI = {10.4007/annals.2012.176.1.3},
       URL = {https://doi.org/10.4007/annals.2012.176.1.3},
}

@article{whyte2004large,
    AUTHOR = {Whyte, Kevin},
     TITLE = {The large scale geometry of the higher {B}aumslag-{S}olitar
              groups},
   JOURNAL = {Geom. Funct. Anal.},
  FJOURNAL = {Geometric and Functional Analysis},
    VOLUME = {11},
      YEAR = {2001},
    NUMBER = {6},
     PAGES = {1327--1343},
      ISSN = {1016-443X,1420-8970},
   MRCLASS = {20F65 (20F69)},
 
MRREVIEWER = {Lee\ Mosher},
       DOI = {10.1007/s00039-001-8232-6},
       URL = {https://doi.org/10.1007/s00039-001-8232-6},
}

@article{Erschler,
    AUTHOR = {Erschler, Anna},
     TITLE = {Generalized wreath products},
   JOURNAL = {Int. Math. Res. Not.},
  FJOURNAL = {International Mathematics Research Notices},
      YEAR = {2006},
     PAGES = {Art. ID 57835, 14},
      ISSN = {1073-7928,1687-0247},
   MRCLASS = {05C25 (60B15)},

MRREVIEWER = {Tatiana\ Smirnova-Nagnibeda},
       DOI = {10.1155/IMRN/2006/57835},
       URL = {https://doi.org/10.1155/IMRN/2006/57835},
}

@article{DL,
    AUTHOR = {Diestel, Reinhard and Leader, Imre},
     TITLE = {A conjecture concerning a limit of non-{C}ayley graphs},
   JOURNAL = {J. Algebraic Combin.},
  FJOURNAL = {Journal of Algebraic Combinatorics. An International Journal},
    VOLUME = {14},
      YEAR = {2001},
    NUMBER = {1},
     PAGES = {17--25},
      ISSN = {0925-9899,1572-9192},
   MRCLASS = {05C25 (20F65)},

MRREVIEWER = {Wolfgang\ Woess},
       DOI = {10.1023/A:1011257718029},
       URL = {https://doi.org/10.1023/A:1011257718029},
}

@article{bertacchi,
    AUTHOR = {Bertacchi, Daniela},
     TITLE = {Random walks on {D}iestel-{L}eader graphs},
   JOURNAL = {Abh. Math. Sem. Univ. Hamburg},
  FJOURNAL = {Abhandlungen aus dem Mathematischen Seminar der
              Universit\"{a}t Hamburg},
    VOLUME = {71},
      YEAR = {2001},
     PAGES = {205--224},
      ISSN = {0025-5858,1865-8784},
   MRCLASS = {60G50 (05C25)},

       DOI = {10.1007/BF02941472},
       URL = {https://doi.org/10.1007/BF02941472},
}

@article{GolanSapir,
    AUTHOR = {Golan, Gili and Sapir, Mark},
     TITLE = {Divergence functions of {T}hompson groups},
   JOURNAL = {Geom. Dedicata},
  FJOURNAL = {Geometriae Dedicata},
    VOLUME = {201},
      YEAR = {2019},
     PAGES = {227--242},
      ISSN = {0046-5755,1572-9168},
   MRCLASS = {20F65},

MRREVIEWER = {Panos\ Papasoglu},
       DOI = {10.1007/s10711-018-0390-x},
       URL = {https://doi.org/10.1007/s10711-018-0390-x},
}

@article{Houghton,
    AUTHOR = {Houghton, C. H.},
     TITLE = {The first cohomology of a group with permutation module
              coefficients},
   JOURNAL = {Arch. Math. (Basel)},
  FJOURNAL = {Archiv der Mathematik},
    VOLUME = {31},
      YEAR = {1978/79},
    NUMBER = {3},
     PAGES = {254--258},
      ISSN = {0003-889X,1420-8938},
   MRCLASS = {20J10},
  
MRREVIEWER = {Kenneth\ S.\ Brown},
       DOI = {10.1007/BF01226445},
       URL = {https://doi.org/10.1007/BF01226445},
}

@article{commH,
    AUTHOR = {Burillo, Jos\'{e} and Cleary, Sean and Martino, Armando and
              R\"{o}ver, Claas E.},
     TITLE = {Commensurations and metric properties of {H}oughton's groups},
   JOURNAL = {Pacific J. Math.},
  FJOURNAL = {Pacific Journal of Mathematics},
    VOLUME = {285},
      YEAR = {2016},
    NUMBER = {2},
     PAGES = {289--301},
      ISSN = {0030-8730,1945-5844},
   MRCLASS = {20F65 (20B07 20F28)},
 
MRREVIEWER = {Enric\ Ventura Capell},
       DOI = {10.2140/pjm.2016.285.289},
       URL = {https://doi.org/10.2140/pjm.2016.285.289},
}

@article{sisto2012metric,
  title={On metric relative hyperbolicity},
  author={Sisto, Alessandro},
  journal={arXiv:1210.8081},
  year={2012}
}

@article{brady2021divergence,
    AUTHOR = {Brady, Noel and Tran, Hung Cong},
     TITLE = {Divergence of finitely presented groups},
   JOURNAL = {Groups Geom. Dyn.},
  FJOURNAL = {Groups, Geometry, and Dynamics},
    VOLUME = {15},
      YEAR = {2021},
    NUMBER = {4},
     PAGES = {1331--1361},
      ISSN = {1661-7207,1661-7215},
   MRCLASS = {20F65 (20F69)},
  
MRREVIEWER = {Michael\ Hull},
       DOI = {10.4171/ggd/632},
       URL = {https://doi.org/10.4171/ggd/632},
}

@article{Parry,
    AUTHOR = {Parry, Walter},
     TITLE = {Growth series of some wreath products},
   JOURNAL = {Trans. Amer. Math. Soc.},
  FJOURNAL = {Transactions of the American Mathematical Society},
    VOLUME = {331},
      YEAR = {1992},
    NUMBER = {2},
     PAGES = {751--759},
      ISSN = {0002-9947,1088-6850},
   MRCLASS = {20F32 (05C25 20E22)},

MRREVIEWER = {John\ D. P. Meldrum},
       DOI = {10.2307/2154138},
       URL = {https://doi.org/10.2307/2154138},
}

@article {Behrstock,
    AUTHOR = {Behrstock, Jason A.},
     TITLE = {Asymptotic geometry of the mapping class group and
              {T}eichm\"{u}ller space},
   JOURNAL = {Geom. Topol.},
  FJOURNAL = {Geometry and Topology},
    VOLUME = {10},
      YEAR = {2006},
     PAGES = {1523--1578},
      ISSN = {1465-3060,1364-0380},
   MRCLASS = {20F69 (20F67 30F60 32G15 57M07)},
  
MRREVIEWER = {Gabriela\ Schmith\"{u}sen},
       DOI = {10.2140/gt.2006.10.1523},
       URL = {https://doi.org/10.2140/gt.2006.10.1523},
}

@article{gen-tes,
title={Lamplighter-like geometry of groups},
  author={Genevois, Anthony and Tessera, Romain},
  journal={arXiv:2401.13520},
  year={2024}
}

@article{ferragut-thesis,
  title={Geometric rigidity of quasi-isometries in horospherical products},
  author={Ferragut, Tom},
  journal={arXiv preprint arXiv:2211.04093},
  year={2022}
}

@article{ferragut-paper,
  title={Geodesics and visual boundary of horospherical products},
  author={Ferragut, Tom},
  journal={arXiv preprint arXiv:2009.04698},
  year={2020}
}

@article{woess-survey,
title={What is a horocyclic product, and how is it related to lamplighters?},
author={Woess, Wolfgang},
journal={Internat. Math. Nachrichten of the Austrian Math. Soc.},
volume={224},
pages={1--27},
year={2013}
}

\end{document}